\documentclass[12pt]{amsart}     

\usepackage{amsfonts}
\usepackage{amsmath}
\usepackage{graphicx}
\usepackage{amsfonts} \usepackage{amsmath, amsthm}
\usepackage{epsfig} \usepackage{verbatim}
\usepackage{amssymb, amsthm, amsmath, bm, tikz, enumerate, color}
\usepackage{ bm, tikz, enumerate}
\usepackage{amsfonts}
\usepackage{amsmath}
\usepackage{graphicx}
\usepackage{amsfonts} \usepackage{amsmath, amsthm}
\usepackage{epsfig} \usepackage{verbatim}
\usepackage{amssymb, amsthm, amsmath, bm, tikz, enumerate, color}
\usepackage{cite, comment}
\usepackage{xcolor}
\usepackage{tikz-cd} 
\usetikzlibrary{spy}

\newtheorem{theorem}{Theorem}[section]

\newtheorem{corollary}[theorem]{Corollary}

\newtheorem{definition}[theorem]{Definition}

\newtheorem{lemma}[theorem]{Lemma}
\newtheorem*{notation}{Notation}

\newtheorem{proposition}[theorem]{Proposition}
\newtheorem{remark}[theorem]{Remark}

\newcommand{\vol}{\mathop{\mathrm{vol}}\nolimits}

\newcommand{\eps}{{\varepsilon}}

\newcommand{\fg}{{\mathfrak{g}}}
\newcommand{\tc}{{\tilde c}}
\newcommand{\hc}{{\hat c}}
\newcommand{\cC}{{\mathcal{C}}}

\numberwithin{equation}{section}

\newcommand{\boundellipse}[3]
{(#1) ellipse (#2 and #3)
}

\author{Dmitry Dolgopyat and P\'eter N\'andori}
\title[Mixing and LLT for hyperbolic flows]
{On mixing and the local central limit theorem for hyperbolic flows}

\address{University of Maryland Department of Mathematics 
4176 Campus Drive College Park, MD 20742-4015 USA}

\begin{document}

\maketitle

\begin{abstract}
We formulate abstract conditions under which a suspension flow satisfies 
the local central limit theorem. We check the validity of these conditions for several systems 
including reward renewal processes, Axiom A flows, as well as the systems admitting Young tower,
such as Sinai billiard with finite horizon,
suspensions over Pomeau-Manneville maps, and geometric Lorenz attractors.
\end{abstract}

\section{Introduction}

Various stochastic aspects of deterministic, chaotic dynamical systems have been extensively 
studied in the last decades.
The central limit theorem (CLT) is a famous example. 
However, its local version (LCLT) has been studied much less for maps and especially for flows. 
In many cases, it is useful to view a chaotic flow as a suspension over a base map, whose chaotic 
properties are easier to prove. 
Then one tries to lift these statements from 
the map to the flow. 
This approach has been applied several times in the dynamics literature 
(see e.g. \cite{R73, DP84, MT04}),
but not so much specifically for the LCLT (we are only aware  of \cite{DN16, AN17}).
This is what this paper is about.

Some special cases, where the LCLT are proved for hyperbolic flows are

(i) Axiom A flows under a non-arithmeticity condition for observables 
\cite{W96};

(ii) the LCLT was obtained in \cite{I08} for a class of flows
whose transfer operator has a spectral gap on a suitable space;

(iii) the free path for Sinai billiard flow with finite horizon 
\cite{DN16}.


In the present paper, we formulate a set of abstract conditions which imply the LCLT for suspension flows.
The most important assumption is the LCLT for the base map. We also discuss connections of LCLT to mixing.
Namely, we give a necessary and sufficient condition for mixing of suspension flows
satisfying LCLT and also prove a joint extension
of mixing and the LCLT (which we abbreviate as MLCLT). We check that the conditions imposed on
the base map are satisfied by a large class of systems
where a Young tower with sufficiently fast return time can be constructed. In particular we generalize 
the results (i) and (iii) mentioned in  above.
Our approach is different from the methods of \cite{W96} and \cite{I08}. In fact we extend the method of
\cite{DN16} to a more general setting.

To be a little more precise, we will work with
a metric space $(X,d)$ with Borel algebra $ \mathcal G$, a probability measure $\nu$ 
and $T$, a self-map of $X$ that preserves $\nu$.  
Let $\tau: X \rightarrow \mathbb R_+$ be an integrable roof function, bounded away from zero. 
Let $\Phi^t$ be the corresponding 
suspension (semi-)flow on the phase space 
$\Omega =\{ (y, s): y \in X, s \in [0, \tau(y)] \} / \sim$, where
$(y,\tau(y)) \sim (Ty,0)$. The induced invariant measure is $\mu = \nu \otimes Leb / \nu(\tau)$.
Let $\bm \varphi : \Omega \rightarrow \mathbb R$ be 
a function.

Setting aside several technical conditions, the LCLT
for flows can be informally stated as follows. 
\begin{equation}
\label{LLTProj}
t^{1/2} \mu \left( \bm x: \int_0^t \bm \varphi(\Phi^s(\bm x)) ds \in w \sqrt t + I \right) \sim \fg(w) u(I)
\end{equation}
as $t \gg 1$, where 
 $I$ is a bounded subinterval of $\mathbb R$, $\fg$ is a Gaussian density
and $u$ is a measure having a large symmetry group. The case $u = Leb$ is called continuous. 
The case where $u$ is invariant under a sublattice $a\mathbb{Z}$
is called discrete.
Several previous papers on LCLT for dynamical systems presented conditions for $u$
to be equal to the countying measure on $a\mathbb{Z}.$ We allow more general measures since it
increases the domain of validity of our results. On the other hand the conditions presented 
in our paper do not guarantee that $a\mathbb{Z}$ is the largest group of translations preserving $u$
since the factor measure on $\mathbb{R}/a \mathbb{Z}$ could ``accidentally'' have some axtra symmetries.

We note that
the LCLT clearly implies the CLT (assuming that the convergence is uniform for $w$ in compact intervals). 

MLCLT is a joint generalization of the the LCLT and the mixing of the flow $(\Omega, \nu, \Phi^t)$.
Recall that mixing means that for all measurable sets $A, B$
$$ \mu(\bm x\in A, \Phi^t \bm x\in B)\sim \mu(\bm x\in A)\mu(\Phi^t \bm x\in B)=\mu(A)\mu(B)\text{ as }t\to\infty. $$
Thus the natural definition of MLCLT would be 
\begin{equation}
\label{NaiveLLT}
\mu\left(x\in A, \Phi^t x\in B, \int_0^t \bm \varphi(\Phi^s(\bm x)) ds \in w \sqrt t + I \right)
\end{equation}
$$ \sim t^{-1/2} \mu(A)\mu(B) 
\fg(w) u(I)
\text{ as }t\to\infty. $$
In the continuous case we indeed define MLCLT by \eqref{NaiveLLT} while in the discrete case
we require that \eqref{NaiveLLT} holds after a certain change of variables which straigthens
$u$ in the fibers, see Definition \ref{defMLCLT} for details. 

In the continuous case, under an assumption that $\bm \varphi$ has zero mean,
 \eqref{NaiveLLT} can be interpreted as mixing property of 
flow $\bar\Phi$ acting on $\Omega\times \mathbb{R}$ by
$ \bar\Phi^t(x, z)=\left(\Phi^t x,   z+\int_0^t \bm \varphi(\Phi^s(\bm x)) ds\right)$
with respect to an infinite invariant measure $\mu\times$Leb.
Also, in the discrete case, MLCLT has an interpretation as mixing of a certain $\mathbb{Z}$
extension of $\Phi$ (see \cite{AN17}).

The result of our analysis is that \eqref{LLTProj} may in general fail for some arithmetic reasons
(see Section \ref{SS:iid} for an explicit example).
However all limit points of the LHS of \eqref{LLTProj} are of the form given by the RHS
of \eqref{LLTProj} with, possibly, different measures $u.$ We also provide sufficient conditions
for \eqref{LLTProj} as well as for MLCLT to hold.

The remaining part of the paper consists of five sections. 
Sections \ref{sec2} and~\ref{sec3} study suspension flows over abstract spaces under the assumptions
that the ergodic sums satisfy the MLCLT. Section \ref{sec2} provides a characterization of mixing in this setting
while Section \ref{sec3}  
contains some abstract assumptions implying various
versions of the MLCLT.
In Sections \ref{secpoly} and \ref{ScHYT}
we  verify our abstract assumptions for certain systems
admitting Young towers. Section \ref{secpoly} deals with expanding Young towers and allows
to obtain MLCLT for suspension flows over non-invertible systems.
In Section \ref{ScHYT} we extend the results of Section \ref{secpoly}
to invertible systems admiting a Young tower. 
In Section \ref{secex} we present several 
specific examples satisfying our assumptions: reward renewal processes, Axiom A flows, Sinai billiard with finite horizon,
suspensions over Pomeau-Maneville maps and geometric Lorenz attractors.

\section{Local Limit Theorem and Mixing.}
\label{sec2}

\subsection{Definitions}

We start with some notations and definitions.

\begin{notation}
\begin{enumerate}
 \item $S_{f} (n,x) = \sum_{k=0}^{n-1} f (T^k x)$ for some function $f: X \rightarrow \mathbb R^d$.
 \item $Leb_{d}$ denotes the $d$ dimensional Lebesgue measure. 
 \item $\fg_{\Sigma}$ denotes the centered Gaussian density with covariance matrix $\Sigma$.
 \item Given a closed subgroup $V$ of $\mathbb R^d$,  $u_V$ denotes the Haar measure on 
$V$ (Lebesgue times counting measure), normalized so that 
$$u(\{ z: \|z\| \leq R \}) 
\sim Leb_{d}(\{ z: \|z\| \leq R \})\text{ as }
R \rightarrow \infty. $$
\end{enumerate}
\end{notation}

\begin{definition}
\label{def:vague}
 Let $\rho_n$ and $\rho$ be 
locally finite measures on some 
Polish space $E$.
Then $\rho_n$
converges vaguely to $\rho$ if $\rho_n(f) \rightarrow \rho(f)$ 
for every compactly supported continuous $f$ (or equivalently 
$\rho_n(\mathcal H) \rightarrow \rho(\mathcal H)$ for every $\mathcal H \subset E$
with $\rho(\partial \mathcal H)=0$). If $f$ is only assumed to be bounded and continuous, then the convergence
is called weak convergence.
\end{definition}

\begin{definition}
\label{def:MLCLT}
We say that $(T,f)$ satisfies the MLCLT ($f: X \rightarrow \mathbb R^d$
is square integrable) if there are some functions $g$ and $h$, where $h$ is bounded and $\nu$-almost everywhere
continuous, such that
\begin{equation}
\label{deccoh}
f = g - h + h \circ T, 
\end{equation}
a closed subgroup $M$ of $\mathbb R^d$, a translation $r \in \mathbb R^d / M$
and a positive definite matrix $\Sigma$ 
such that, as $n \rightarrow \infty$, the following holds for any bounded and
continuous $\mathfrak x, \mathfrak y: X \rightarrow \mathbb R$ and for any continuous
and compactly supported $\mathfrak z : \mathbb R^d \rightarrow \mathbb R,$
for any sequence $w_n$ satisfying
$$w_n \in M+ nr, \| w_n - w \sqrt n \| \leq K$$
we have
\begin{equation}
\label{MLCLTtestf}
n^{d/2} \int \mathfrak x (x) \mathfrak y(T^n x) \mathfrak z (S_{g}(n,x) - n \nu(g) -w_n) d\nu(x) 
\end{equation}
$$\rightarrow \fg_{\Sigma} (w) \int \mathfrak x d\nu \int \mathfrak y d\nu \int \mathfrak z d u_M $$
and the convergence is uniform once $K$ is fixed and $w$ is chosen from a compact set. 
\end{definition}

\begin{definition}
\label{def:nd}
We say that $f:X \rightarrow \mathbb R^d$ is non-degenerate if for any $g$ cohomologous to $f$ (i.e. satisfying
\eqref{deccoh} with some measurable $h$), 
the minimal translated subgroup supporting $g$ is $d$-dimensional.
That is, if $L$ is subgroup of $\mathbb R^d$
and $a\subset \mathbb R^d$ is a vector such that
$g\in a+L$ then $L$ is  $d$-dimensional.
\end{definition}

We remark that a slight generalization of the MLCLT would allow dim$(M) < d$ and would naturally accommodate
degenerate observables. We do not consider this case here.

If $f$ satisfies the MLCLT with $h=0$, we say that $f$ is {\it minimal} (this is the case e.g. if $M = \mathbb R^d$). 
For non-minimal functions, the limit
may not be a product measure, but we have instead the following lemma which is a consequence of the continuous 
mapping theorem (recall that $h$ is bounded and almost everywhere continuous).

\begin{lemma}
\label{lemma:nonmin}
Assume that $(T,f)$ satisfies the MLCLT. Then with the notation of Definition \ref{def:MLCLT}, 
\begin{align*}
& \lim_{n\rightarrow \infty}  n^{d/2} \int \mathfrak x (x) \mathfrak y(T^n x) \mathfrak z (S_{f}(n,x) - n \nu(f) -w_n) d\nu(x) \\
& = 
\fg_{\Sigma} (w) \int 
\mathfrak x (x) \mathfrak y (y) \mathfrak z (z - h(x) + h(y)) 
d(\nu(x)  \times u_M(z) \times \nu(y)).
\end{align*}
\end{lemma}

Observe that by Lemma \ref{lemma:nonmin} if $(T,f)$ satisfies the 
MLCLT, then $M$ and $r$ are uniquely defined. Indeed, if $f = g_1 - h_1 + h_1 \circ T
= g_2 - h_2 + h_2 \circ T$ with $g_1, g_2$ minimal, 
then applying Lemma \ref{lemma:nonmin} with $g_1$, $g_2$ and $\mathfrak x=\mathfrak y=1$, we find
$M_i + r_i = supp( (u_{M_{3-i}} + r_{3-i})*dist(h_{3-i})*dist(h_{3-i})) \supset 
M_{3-i} + r_{3-i}$ for $i=1,2$,
where $dist$ means distribution and $*$ is convolution. We will use the notation $M(f)$ and $r(f)$.

 \begin{remark}
 \label{remsets}
By standard arguments concerning vague convergence (sometimes called Portmanteau theorem) one can
give equivalent formulations of Definition \ref{def:MLCLT} and Lemma \ref{lemma:nonmin}.
We will use these versions for convenience.
 
 Definition \ref{def:MLCLT}  remains unchanged if we choose $\mathfrak x, \mathfrak y $
 to be indicators of sets $A,B$ 
 whose boundary has $\nu$--measure zero and $\mathfrak z$ 
 to be an indicator
 of a bounded set $\mathcal H$ whose boundary, w.r.t. the topology on $M$ has 
 $u_M$--measure zero.
For fixed $A,B,$  we can think about the MLCLT as 
 vague convergence of measures.

Similarly, Lemma \ref{lemma:nonmin}  remains valid if we consider the indicator 
test functions as above, now also satisfying
\begin{equation*}
\nu(\partial A) = \nu(\partial B) = (\nu \times \nu) ((x,y): u_M (\partial (\mathcal H + h(x) - h(y) \cap M)) >0 ) = 0.
\end{equation*}
 \end{remark}

\subsection{Characterization of mixing.}

We impose the following hypotheses.

\begin{enumerate}
 \item[(\bf H1)] $(T, \tau)$ satisfies the MLCLT.
 \item[(\bf H2)] ({\em Moderate deviation bounds})
For some (and hence for all) $R$ large enough, we have
$$
\lim_{K \rightarrow \infty} \limsup_{ w \rightarrow \infty} 
\sum_{n:  | n-w| > K \sqrt{ w}} \nu (x: S_{\tau}(n,x) \in [w \nu(\tau)- R, w \nu(\tau)+ R]) 
= 0,
$$
\end{enumerate}

\begin{proposition}
\label{prop1}
Assume {\bf (H1)} and  {\bf (H2)}. Then the following are equivalent
\begin{enumerate}
\item[(a)] $\Phi$ is weakly mixing 
\item[(b)] $\Phi$ is mixing 
\item[(c)] either $M(\tau) =\mathbb R$ or $M(\tau) = \alpha \mathbb Z$ and $r(\tau) / \alpha \notin 
\mathbb Q$.
\end{enumerate}
\end{proposition}

Clearly (b) implies (a). To prove the proposition 
we first show in Section \ref{SSNWM} that if (c) fails  then the flow is not weak mixing,  
and then we show  in Section \ref{SSNAMix} that (c) implies~(b).

\subsection{Non-mixing case.}
\label{SSNWM}
Assume that either $\tau$ is a coboundary or  $M(\tau) = \alpha \mathbb Z$ and $r/\alpha \in \mathbb Q$, 
where $r= r(\tau)$.
In both cases, there is some $h_{\tau} $, some rationally related numbers $\alpha, r$ and a 
function $g: X \rightarrow \mathbb Z$ such that
\begin{equation}
\label{eq:50}
\tau(x) = h_{\tau}(x) - h_{\tau}(Tx) + r + \alpha g(x) 
\end{equation}
(If $\tau$ is a coboundary, then $\alpha = 0$.) Note that $h_{\tau}$ 
is defined up to an additive constant. 
Choose this constant in such a way that $\nu (Y) >0$, where 
$Y=\{ x: 0<h_{\tau}(x)<\tau(x)\}$ and define 
$C = \{(x,t): t = h_{\tau}(x) \in [0, \tau(x)]\}$.
Let $\varsigma:Y \rightarrow \mathbb R_+$ be defined by
$$\varsigma (x) = \min_{s>0} \{ \Phi^s (x,h_{\tau}(x)) \in C\} $$
Then \eqref{eq:50} gives
$$
\varsigma (x) = -nr + \alpha \sum_{i=1}^n g(T^{i-1}x)
$$ 
where
$n=n(x)$ is the number of hits of the roof before time $\varsigma(x)$.
Thus $\varsigma(x) \in  G$, the group generated by $r$ and $\alpha$. 
Since $\alpha$ and $r$ are rationally related,
there is some $b>0$ such that $G = b \mathbb Z$. 
Let us fix some $\delta \in (0, b/2)$ and write
$c_{\delta} = \{(x,t): h_{\tau}(x) \in [0, \tau(x)], |t-h_{\tau}(x)| < \delta\}$. 
By construction, $\mu (C_{\delta} \cap \Phi^{-t}C_{\delta}) = 0$ whenever 
$t \in k b + (\delta, b - \delta)$. This shows that $\Phi$ is not weakly mixing.

\subsection{Mixing case}
\label{SSNAMix}
Assume that $M(\tau) = \mathbb R$ or $M(\tau) = \alpha \mathbb Z$
and the shift $r=r(\tau)$ associated with $\tau$ satisfies $r \notin \mathbb Q(\alpha)$. 
We use the formulation of LCLT given in Remark \ref{remsets}.

It is enough to show that 
\begin{equation}
\label{eq:mixing}
 \lim_{t \rightarrow \infty} \mu(\mathcal A \cap \Phi^{-t} \mathcal B)  = \mu(\mathcal A) \mu(\mathcal B).
\end{equation}
in case $\mathcal A = A \times I$, $\mathcal B = B \times J$.
Decompose
$$ \tau (x) = \hat \tau (x) - h_{\tau}(x) + h_{\tau}(Tx)$$
where $h_\tau$ is given in \eqref{deccoh}.
Let us fix some $s \in I$ and write
\begin{equation}
\label{DefNu}
N_{u} = N_{u}(x) = \max \{ n: S_{\tau} (n,x) \leq u\}.
\end{equation}
Then
\begin{eqnarray*}
&&\mu(\mathcal A \cap \Phi^{-t} \mathcal B)  \\
&=&\frac{1}{\nu(\tau)}\int_{s \in I} \nu \bigg(  
x \in  A: S_{\tau}(N_{t+s},x) -t \in - J +s, T^{N_{t+s}(x,s)} (x) \in B \bigg) ds\\
 \end{eqnarray*}
For a fixed $s$, let $C(s)$ be the set of points $(x,z,y) \in X \times \mathbb R \times X$ that satisfy
\begin{enumerate}
 \item $x \in A$,
 \item $y \in B$,
 \item $z \in - J +s + h_{\tau}(x) - h_{\tau}(y)$.
\end{enumerate}
Then we have
\begin{eqnarray*}
&&\mu(\mathcal A \cap \Phi^{-t} \mathcal B) \\
&=& \int_{s \in I} \frac{1}{\nu(\tau)} \nu \bigg(  
x:
\big(
x, 
S_{\hat \tau}(N_{t+s},x) -t, 
T^{N_{t+s}} (x)
\big) \in C(s)
 \bigg) ds
\end{eqnarray*}
Let 
 $$
 C_n(s) = \{ x: 
 \big(
x, 
S_{\hat \tau}(n,x) -t, 
T^{n} (x)
\big) \in C(s) \}.
 $$
Observe that $x \in C_n(s)$ implies $t+s - S_{\tau}(n,x) \in J$ and consequently $N_{t+s}(x) = n$.
Thus 
$$
\mu(\mathcal A \cap \Phi^{-t} \mathcal B) = 
 \int_{s \in I} \frac{1}{\nu(\tau)} \nu (x: x \in C_{N_{t+s}} (s)) ds = 
\int_{s \in I}  \sum_{n=1}^{t/\inf \tau } \frac{1}{\nu(\tau)} \nu (C_{n} (s)) ds
$$
We write the above sum as $S_1+ S_2$, where
$$
S_1 = \sum_{n: |n-t/\nu(\tau)| < K \sqrt t} \int ...
$$
with $K \gg 1$ and $S_2$ is an error term (which is small by {\bf (H2)}). Now let us apply
{\bf (H1)} (and Lemma \ref{lemma:nonmin}) to compute $S_1$.

Assume first that $M(\tau) = \mathbb R$. Observe that
$(\nu \times Leb_1 \times \nu)(C(s)) = \nu(A) |J| \nu(B)$.
Thus {\bf (H1)} and dominated convergence give that for any fixed $K$
\begin{equation}
\label{eq:Rie00}
S_1 \sim \sum_{n: |n-t/\nu(\tau)| < K \sqrt t} \frac{1}{\nu(\tau) \sqrt n} 
\fg_{\sigma} \left( \frac{m}{\sqrt n} \nu(\tau) \right)
\nu(A) |J| \nu(B) |I|,
\end{equation}
where $\sigma$ is the variance of $\tau$ and $m = \lfloor t/\nu(\tau) \rfloor -n$. 
Substituting a Riemann sum with the Riemann integral, we find that 
$$ S_1 \sim \mu( \mathcal A) \mu( \mathcal B)(1+o_{K \rightarrow \infty}(1)). $$
On the other hand by {(\bf H2}) $S_2\to 0$ as $K\to\infty$ uniformly in $n.$
This proves \eqref{eq:mixing}.

In case $M(\tau) = \alpha \mathbb Z$, $r \notin \mathbb Q(\alpha)$ we apply a similar approach with the 
following differences. Note that \eqref{eq:Rie00} is replaced by
\begin{equation}
\label{eq:Rie000}
\sum_{n: |n-t/\nu(\tau)| < K \sqrt t} \frac{1}{\nu(\tau) \sqrt n} 
\fg_{\sigma} \left(  \frac{m}{\sqrt n} \nu(\tau) \right)
\int_{s \in I} (\nu \times u_{\alpha \mathbb Z} \times \nu) (C(s) + \varkappa_n) ds
\end{equation}
where $\varkappa_n \in \mathbb R /(\alpha \mathbb Z)$ is defined by
$$
\varkappa_n = t - n r  \;(\bmod\; \alpha)
$$
and $C(s) + \varkappa$ is defined as
$$\{ (x,z+\check \varkappa,y): (x,z,y) \in C(s)\},$$
with some $\check \varkappa \in \mathbb R$, $\check \varkappa + \alpha \mathbb Z = 
\varkappa$ (since $u$ is the Haar measure, 
\eqref{eq:Rie000} does not depend on the choice of the representative).
Now writing the sum in \eqref{eq:Rie000} as
$$
\sum_{j=0}^{2 K \sqrt t /N -1} \sum_{m=-K \sqrt t + jN}^{-K \sqrt t + (j+1)N -1},
$$
for some large $N$ and using Weyl's theorem, we conclude that \eqref{eq:Rie00} still holds. 
Then we can complete the proof as before.

\section{From LLT for base map to LLT for flows.}
\label{sec3}

\subsection{Definitions, assumptions}

Given an observable 
$\bm \varphi: \Omega \rightarrow \mathbb R$ let
\begin{equation}
\label{FlowObs-SectionObs}
\check{\bm \varphi}(x) = \int_0^{\tau(y)} \bm \varphi (x,s) ds.
\end{equation}

We impose the following
hypotheses.

\begin{enumerate}
 \item[(\bf H3)] $\Phi$ is mixing.
\item[(\bf H4)]  $\mu (\bm \varphi) =0.$
\item[(\bf H5)] $\bm \varphi$ is bounded and $\mu$-almost everywhere continuous.
\item[(\bf H6)] $(T, (\check{\bm \varphi},\tau ))$ satisfies the MLCLT.
\item[(\bf H7)] ({\em Moderate deviation bounds})
For $f = (\tau,\check{\bm \varphi} )$ and for some 
(and hence for all) $R$ large enough, we have
$$
\lim_{K \rightarrow \infty} \limsup_{ w \rightarrow \infty}  w^{1/2} 
\sum_{n:  | n-w| > K \sqrt{ w}} \nu (x: S_f(n,x) \in B(w \nu(f), R)) = 0,
$$
where $B(v,R)$ is the ball of radius $R$ centered at $v \in \mathbb R^d$ 
\end{enumerate}

 Now we define the MLCLT for the flow for $d$ dimensional observables 
 (most of this paper is about the case $d=1$).

\begin{definition}
\label{defMLCLT}
We say that $(\Phi, \bm \varphi)$ ($\bm \varphi: \Omega \rightarrow \mathbb R^d$ square integrable)
satisfies the MLCLT if
there exists some closed subgroup 
$\mathcal V \subset \mathbb R^d$, $R \in \mathbb R^d / \mathcal  V$,
a $\mu$-almost everywhere continuous
function $H: \Omega \rightarrow \mathbb R^d $, bounded on $\{ (x,s)\in \Omega, s \leq M \}$ for all $M >0$
and a positive definite matrix 
$\Sigma = \Sigma(\bm \varphi)$ such that, as $t \rightarrow \infty$, the following holds for any $M >0$, any bounded and
continuous $\mathfrak X, \mathfrak Y: \{ (x,s)\in \Omega, s \leq M \} \rightarrow \mathbb R$ and for any continuous
and compactly supported $\mathfrak Z : \mathbb R^d \rightarrow \mathbb R$:
\begin{align*}
t^{d/2} \int \mathfrak X (x,s) \mathfrak Y( \Phi^t (x,s)) \mathfrak Z \left( \mathcal F(x,s) \right) d\mu(x,s) 
\rightarrow \fg_{\Sigma} (W) \int \mathfrak X d\mu \int \mathfrak Y d\mu \int \mathfrak Z d u_{\mathcal  V}.
\end{align*}
Here, 
\begin{equation}
\label{defF}
\mathcal F(x,s) =
\int_0^t \bm \varphi(\Phi^{s'}(x,s))ds' + H(x,s) - H(\Phi^t(x,s)) - W(t),
\end{equation}
and
$W(t)$ is assumed to satisfy
\begin{equation}
\label{condMLCLT}
W(t) \in \mathcal V + Rt, | W(t) - W \sqrt t | \leq K.
\end{equation}
We also require that the convergence in uniform once $K$ is fixed and $W$ is chosen from a compact set. 
\end{definition}
 
 We can get rid of the coboundary term similarly to the case of the map 
(cf. Lemma \ref{lemma:nonmin}).

\begin{lemma}
\label{lemma:nonminH0}
Assume that $(\Phi, \bm \varphi)$ 
satisfies the MLCLT. Then with the notation of Definition \ref{defMLCLT},
\begin{align}
& \lim_{t \rightarrow \infty} t^{d/2} 
\int \mathfrak X (x,s) \mathfrak Y( \Phi^t (x,s)) 
\mathfrak Z \left( \int_0^t \bm \varphi(\Phi^{s'}(x,s))ds'   - W(t) \right) d\mu(x,s) \nonumber \\
& = 
\fg_{\Sigma} (W) \int
\mathfrak X (x,s) \mathfrak Y( y,s')
\mathfrak Z (z - H(x,s) + H(y,s') )
d(\mu(x,s)  \times u_{\mathcal V}(z) \times \mu(y,s')). \label{MLCLTnonmimH0}
\end{align}
\end{lemma}

Before proceeding to the MLCLT for the flow, we make some 
important remarks.

\begin{remark}
As Lemma \ref{lemma:nonminH0} shows, we can assume that 
$H$ take values in $\mathbb R^d / \mathcal V$.
\end{remark}

\begin{remark}
\label{bddtau}
Definition \ref{defMLCLT} becomes simpler if $\tau$ is bounded. Indeed, in this case
 $\mathfrak X$ and $\mathfrak Y$  are just any bounded and continuous functions on $\Omega$. (In general, $X$ is not compact.)
If $\tau$ is unbounded, we only consider the test functions given in Definition \ref{defMLCLT}
in the first (abstract) part of the paper. However, later we will extend the convergence to any
bounded and continuous $\mathfrak X$ and $\mathfrak Y$ in important applications 
(namely, first return map to the base in Young towers, see Proposition \ref{propunbdH}). 
\end{remark}
 
 \begin{remark}
 \label{remsets2}
 Similarly to Remark \ref{remsets} we have the following reformulations of 
Definition \ref{defMLCLT} and Lemma 
 \ref{lemma:nonminH0}.  
 
Definition \ref{defMLCLT} remains unchanged if we choose 
 $\mathfrak X, \mathfrak Y $
 to be indicators of sets $\mathcal A, \mathcal B$)
 whose boundary has $\mu$--measure zero and 
 $\mathfrak Z$ to be an indicator
 of a bounded set $\mathcal H$ whose boundary, w.r.t. the topology on $\mathcal V$ has 
 $u_{\mathcal V}$--measure zero.
 Furthermore, it suffices to consider indicators of product sets
 $\mathfrak X = 1_{\mathcal A}$, where $\mathcal A = A \times I$, $\nu( \partial A) = 0$ and $I$ is a subinterval
 of $[0, \inf \{ \tau(x), x \in A\}]$. Similarly, we can choose $\mathfrak Y = 1_{\mathcal B}$, where 
 $\mathcal B = B \times J$,  $\nu( \partial B) = 0$ and $J$ is a subinterval
 of $[0, \inf \{ \tau(x), x \in B\}]$. 

Also Lemma \ref{lemma:nonminH0} remains valid if we consider the indicator 
test functions as above, satisfying
\begin{equation*}
\nu(\partial A) = \nu(\partial B) =
\mu \times \mu ((x,s,y,s'): u_{\mathcal V} (\partial (\mathcal H + H(x,s) - H(y,s') \cap \mathcal V)) >0 ) = 0.
\end{equation*}
 \end{remark}

\subsection{The linearized group}
Consider 
$M = M(\tau,\check{\bm \varphi} )$ and $r = r(\tau,\check{\bm \varphi} )$. 
The {\it linearized group} of $(\tau,\check{\bm \varphi} )$ is the closure of the group
generated by $M$ and $r$. We denote this group by $\hat{\mathcal V}$.
Define $\varkappa_n \in \mathbb R^2/ M$ by
$$
\varkappa_n = - n r   \;(\bmod\; M)
$$
and write $ M= Y \times L$, where $Y$ is a subspace of dimension $d_1$ and 
$L$ is a lattice of dimension $d_2$ with $d_1, d_2 \in \{ 0,1,2\}$, $d_1+d_2 = 2$.
The self map of $\mathbb R^2 / M$, defined by 
$\varkappa \mapsto \varkappa - r$ is linearly conjugate
to a translation of the $d_2$ dimensional torus. Consequently, the closure of any orbit is a subtorus. Furthermore, 
the orbit is uniformly distributed on this subtorus by Weyl's theorem.
We conclude 
\begin{lemma}
\label{lem:Weyl}
$\frac1{N} \sum_{n=1}^{N} u_M(.+ \check \varkappa_n)$ converges weakly to $u_{\hat{\mathcal V}}$
as $N \rightarrow \infty$, where
$\check \varkappa_n \in \mathbb R^2$ satisfies $\check \varkappa_n + M= \varkappa_n$
\end{lemma}

Note that by Proposition \ref{prop1} and by the definition of the linearized group,
the projection of $\hat{\mathcal V}$ to the second coordinate is dense. Consequently, 
under assumptions {\bf (H1)} - {\bf (H7)},
one of the following cases holds.
\begin{enumerate}
 \item[\bf (A)] $\hat{\mathcal V} = \mathbb R^2$
 \item[\bf (B)] $\hat{\mathcal V} = a \mathbb Z \times \mathbb R$
 \item[\bf (C)] $\hat{\mathcal V} = \{ (x_1, x_2)\in \mathbb R^2: x_2 - \alpha x_1 \in \beta \mathbb Z\}$ with some $\beta \neq 0$
 \item[\bf (D)] 
 $\hat{\mathcal V}$ is generated by $(a,b)$ and $(0,d)$, where $b/d$ is irrational. We can assume 
$a,d>0$.
 \item[\bf (E)] $\hat{\mathcal V}$ is generated by $(a',b')$ and $(c',d')$, where 
 $a'/c'$ and $b'/d'$ are irrational. We can assume $d'>0$ and $a'd'-b'c'>0$.
\end{enumerate}

We can interpret cases {\bf (A)} and {\bf (B)} as arithmetic independence between 
$\check{\bm \varphi}$ and $\tau$, as
$\hat{\mathcal V} = \pi_1 \hat{\mathcal V} \times \pi_2 \hat{\mathcal V}$,
where $\pi_i$ is the projection to the $i$th coordinate.

\subsection{MLCLT for $\Phi$}
\label{sec:mainres}

To fix notations, we write the decomposition \eqref{deccoh} for the functions $\check{\bm \varphi}$
and $\tau$ as
\begin{equation}
\label{defh}
\check{\bm \varphi} (x) = \psi(x) - h(x) + h(Tx), \text{ and }
\end{equation}
\begin{equation}
\label{defh_tau}
\tau (x) = \hat \tau (x) - h_{\tau}(x) + h_{\tau}(Tx).
\end{equation}

\begin{theorem}
\label{thm1}
Assume hypotheses {\bf (H1)}-{\bf (H7)}. 
In cases {\bf (A)}-{\bf (C)}, $(\Phi, \bm \varphi)$ satisfies the MLCLT with $\mathcal V, R$ and $H$ given by:
\begin{itemize}
 \item[{\bf (A)}] $\mathcal V = \mathbb R$, $R=0$, $H(x,s) = 0$, 
\item[{\bf (B)}] $\mathcal V = a \mathbb Z$, $R=0$,
$H_B(x,s) = \int_0^{s} \bm \varphi(x,s') ds' + h(x)  \;(\bmod\; \mathcal V)$,  
\item[{\bf (C)}] $\mathcal V = \frac{\beta}{\alpha} \mathbb Z$,
$R= \frac{1}{\alpha}$,
$$H_C(x,s) = \int_0^{s} \bm \varphi(x,s') ds' + h(x)  -
 \frac{1}{\alpha} (s + h_{\tau}(x)) \;(\bmod\; \mathcal V).$$
\end{itemize}
In all the above cases, 
$\Sigma (\bm \varphi) = \frac{1}{\nu(\tau)} (\Sigma(\check{\bm \varphi}, \tau))_{11}$
\end{theorem}

In case {\bf (D)}, a weaker version of Theorem \ref{thm1} holds. Namely, the limiting measure 
depends on $t \;(\bmod\; d)$ and is not a product over $\Omega \times a \mathbb Z \times \Omega$.
Let us denote by $\Xi_{t} = \Xi_{t,W(t),\bm \varphi, H}$
the push forward of the measure $\mu$ by the map
$$
(x,s) \mapsto \left( (x,s), \int_0^t \bm \varphi (\Phi^{s'} (x,s)) ds' +H(x,s) - H(\Phi^t(x,s)) -W(t), \Phi^t(x,s) \right).
$$

For simplicity, we only consider the case when
\begin{enumerate}
\item[{\bf (H8)}] $\tau$ is bounded.
\end{enumerate}

\begin{proposition}
\label{prop:pos}
Assume conditions {\bf (H1)}-{\bf (H8)} and Case {\bf (D)}. 
Recall the notation introduced in Remark \ref{remsets}.Then for any $l \in \mathbb Z$
and any $W(t) \in a \mathbb Z$ with $W(t) \sim W \sqrt t$, we have
$$
\lim_{t\rightarrow \infty} |\sqrt t \Xi_{t,W(t),\bm \varphi,H} (\mathcal A \times \{la \} \times \mathcal B)- 
\mathcal I_t | =0,
$$
where $H= \int_0^{s} \bm \varphi(x,s') ds' + h(x)  \;(\bmod\; a)$,
\begin{equation}
\label{eq:propposgen}
\mathcal I_t = \frac{ad \fg_{\Sigma}(W)}{\nu(\tau)} \sum_{|k| \leq \frac{ \| \tau \|_{\infty} + 2 \| h_{\tau}\|_{\infty} +1}{d}} 
\int_{(x,s) \in \mathcal A} \int_{y \in B} 
1_{\{\rho + kd +h_{\tau}(x) -h_{\tau}(y) \in J \}}
d \nu(y) d\mu(x,s),
\end{equation}
and 
$\rho \in [0,d)$ satisfies
\begin{equation}
\label{defrho}
\rho \equiv s+ t - \left( \frac{W(t)}{a} +l \right) b \;(\bmod\; d). 
\end{equation}

\end{proposition}

Note that in the special case when $\tau$ is minimal and thus $h_{\tau}$ can be chosen to be zero,
the formula \eqref{eq:propposgen} reduces to 
\begin{equation}
\label{eq:propposmin}
\mathcal I_t = \frac{\nu(A)}{\nu(\tau)} 
\fg_{\Sigma}(W) a d \int_{s \in I} \mathrm{Card} (m: \rho + md \in J) ds 
\frac{\nu(B)}{\nu(\tau)}.
\end{equation}
This formula is consistent with 
 Theorem \ref{thm1} {\bf (B)} in the sense that 
for $I$ and $J$ fixed and $d \ll 1$, $d \int_{s \in I} Card (m: \rho + md \in J) ds \approx |I||J|$ and
$u_{a \mathbb Z}$ is $a$ times the counting measure. Thus, recalling Remark \ref{remsets}, we recover 
Theorem \ref{thm1} {\bf (B)} in the limit $d \searrow 0$.

\begin{proposition}
\label{prop:E}
Assume conditions {\bf (H1)}-{\bf (H8)} and Case {\bf (E)}. Then the statement of Proposition \ref{prop:pos}
remains valid with the following changes:
\begin{enumerate}
 \item $a:=a'-\frac{c'd'}{b'}$, $b: = b', d: =d'$,
 \item $W(t) \in a \mathbb Z + \frac{c'}{d'} t$,
 \item $H_E(x,s) = \int_0^{s} \bm \varphi(x,s') ds' + h(x)  - \frac{c'}{d'} (s + h_{\tau}(x)) \;(\bmod\; a).$
 \item \eqref{defrho} has to be replaced by 
 $\rho \equiv s+ t - \left( \frac{W(t)-c't/d'}{a} +l\right)b \;(\bmod\; d). $
\end{enumerate}

\end{proposition}

\begin{remark}
 Case {\bf (C)} and can be reduced to Case {\bf (B)}
and Case {\bf (E)} can be reduced to Case {\bf (D)}
 by applying the shear 
$
\left[
\begin{array}{c c}
 1& -v\\
 0 & 1
\end{array}
\right]
$ to $\hat{\mathcal V}$
with $v= \frac{1}{\alpha}$ in Case {\bf (C)} and  with $v= \frac{c'}{d'}$ in Case {\bf (E)} (this will produce 
a non-zero $R$ in the MLCLT and in its weaker form as in Propositions \ref{prop:pos} and \ref{prop:E}).
We note that while the shear is uniquely determined in Case {\bf (C)}, it is not unique in Case~{\bf (E)}.
\end{remark}

\subsection{Higher dimensions}
Here we state the high dimensional generalization of Theorem \ref{thm1}. We omit the proof as
it is analogous to the proof of the one dimensional case.
We need to replace 
$w^{1/2}$ by $w^{d/2}$ in {\bf (H7)}.  
Observe that we can construct the group $\hat{\mathcal V}$ exactly as before as Weyl's theorem holds in any
dimension.
Now we have the following.

\begin{theorem}
\label{thmhighd}
Assume  {\bf (H1)} - {\bf (H8)}.
Assume furthermore that there is a closed subgroup $\mathcal V \subset \mathbb R^d$, a 
$(d+1)\times (d+1)$ matrix of the form
$
A = \left[
\begin{array}{c c}
I_d &-\bm v\\
\bm 0^T &1
\end{array}
\right]
$, with $\bm v \in \mathbb R^d$ such that $A \hat{\mathcal V} = \mathcal V \times \mathbb R$
(we assume that $\bm v$ is orthogonal to the linear subspace contained in $\mathcal V$, the choice of such $\bm v$
is unique). Then $(\Phi, \bm \varphi)$ satisfies the MLCLT with $\mathcal V$, $R= \bm v + \mathcal V$ 
and
$$H(x,s) = \int_0^{s} \bm \varphi(x,s') ds' + h(x)  - v (s + h_{\tau}(x)) \;(\bmod\; \mathcal V).$$
\end{theorem}


\subsection{Proof of Theorem \ref{thm1}}
\label{sec:casesac}

The proof of Theorem \ref{thm1} is similar to the proof that (c) implies (b) in Proposition
\ref{prop1}. The main difference is that we apply {\bf (H6)} instead of {\bf (H1)}.

Recall the notation introduced in \eqref{defh} and \eqref{defh_tau} and write 
\begin{equation}
\label{Hhat}
\hat H (x,s)  = \int_0^{s} \bm \varphi(x,s') ds' + h(x).
\end{equation}
 Recall \eqref{DefNu}.
 By construction, for any $(x,s) \in \Omega$ and any $t>0$, we have
 \begin{eqnarray}
 && \int_0^t \bm \varphi(\Phi^{s'}(x,s))ds' + \hat H(x,s) - \hat H(\Phi^t(x,s)) \\
 &=& S_{\check{\bm \varphi}} (N_{t+s},x) + h(x) - h (T^{N_{t+s}} x) \nonumber \\
&=&  S_{\psi} (N_{t+s},x)\label{eq:flowtomap}.
 \end{eqnarray}

According to Remark \ref{remsets2}, we choose $\mathfrak X = 1_{\mathcal A}$, 
$\mathfrak X = 1_{\mathcal B}$ 
where $\mathcal A=A\times I$, $\mathcal B=B\times J,$
and $\mathfrak Z  = 1_{\mathcal H}$, where $\mathcal H \subset \mathcal V$.
Without loss of generality, we can assume that $\mathcal H$ is a compact interval in cases {\bf(A)} and
{\bf (B)} and 
$|\mathcal H| = 1$ in case~{\bf (C)}.

Recall the definition of $\Xi_t$ and write $\mathcal C = \mathcal A \times \mathcal H \times \mathcal B$.
We have 
\begin{eqnarray}
\Xi_t(\mathcal C) &=& \mu \left(  
(x,s) \in \mathcal A: S_{\psi}(N_{t+s},x) -W(t)
\in \mathcal H, 
\Phi^t (x,s) \in \mathcal  B \right) \label{XiC1} \\
&=& \frac{1}{\nu(\tau)}\int_{s \in I} \nu \bigg(  
x \in  A: S_{\psi}(N_{t+s},x) - W(t) \in \mathcal H, \label{XiC2} \\
&& S_{\tau}(N_{t+s},x) -t \in - J +s,
T^{N_{t+s}(x,s)} (x) \in B \bigg) ds. \nonumber
\end{eqnarray}
Now for
some fixed $s \in I$, 
we define $C(s)$ to be the set of points $(x, z, y)$, $x, y \in X$, $z = (z_1,z_2) \in \mathbb R^2$ 
that satisfy
$$ x \in A;\quad  y \in B; \quad z_1 \in \mathcal H; \quad
z_2 \in -J + s+ h_{\tau} (x) - h_{\tau} (y).$$
Then we have
$$\Xi_t(\mathcal C) 
= \int_{s \in I} \frac{1}{\nu(\tau)} \nu \bigg(  
x:
\big(
x, 
S_{\psi}(N_{t+s},x) - W(t), 
S_{\hat \tau}(N_{t+s},x) -t, 
T^{N_{t+s}} (x)
\big) \in C(s)
 \bigg) ds
$$
Let 
 $$
 C_n(s) = \{ x: 
 \big(
x, 
S_{\psi}(n,x) - W(t),
S_{\hat \tau}(n,x) -t, 
T^{n} (x)
\big) \in C(s) \}.
 $$
Observe that $x \in C_n(s)$ implies $t+s - S_{\tau}(n,x) \in J.$ Hence $N_{t+s}(x,s) = n$.
Therefore
\begin{equation}
\label{eq:H7'use}
\Xi_t(\mathcal C) =  \int_{s \in I} \frac{1}{\nu(\tau)} \nu (x: x \in C_{N_{t+s}} (s)) ds = 
\int_{s \in I}  \sum_{n=1}^{t/\inf \tau } \frac{1}{\nu(\tau)} \nu (C_{n} (s)) ds.
\end{equation}
By {\bf (H4)}, $\nu(\check{\bm \varphi}, \tau) = (0,\nu(\tau))$.
We write the above sum as $S_1+ S_2$, where
$$
S_1 = \sum_{n: |n-t/\nu(\tau)| < K \sqrt t} \int ...
$$
with $K \gg 1$ and $S_2$ is an error term which is small by {\bf (H7)}. It suffices to compute $S_1$.

Let us first study the special case 
when $M(\check{\bm \varphi}, \tau) = \mathbb R^{2}$. We will refer to this case as {\it non-arithmetic}.
Clearly, the non-arithmeticity implies Case {\bf (A)}.

\subsubsection{The non-arithmetic case}
\label{sec:A1}
We assume
$M(\check{\bm \varphi}, \tau) = \mathbb R^{2}$. 
Clearly, $r=0$, $\mathcal V = \mathbb R$ and $R= 0$.

For a fixed $K$, we have by 
{\bf (H6)} and by dominated convergence that
\begin{equation}
\label{eq:Rie1}
S_1 \sim \sum_{n: |n-t/\nu(\tau)| < K \sqrt t} \frac{1}{\nu(\tau) n} \fg^*
\int_{s \in I} (\nu \times Leb_{2} \times \nu) (C(s)) ds,
\end{equation}
where $m = \lfloor t/\nu(\tau) \rfloor -n$ and 
$\fg^*=\fg_{\sigma} \left(  - W \sqrt{\nu(\tau)}, \frac{m}{\sqrt n} \nu(\tau) \right).$
 
To compute $\int_{s \in I} (\nu \times Leb_{2} \times \nu) (C(s)) ds$, let us evaluate the integral with respect to 
$z_2$ first. For any $s, x, y$ fixed, this integral is equal to $|J|$. Consequently,
$$
\int_{s \in I} (\nu \times Leb_{2} \times \nu) (C(s)) ds = |I| \nu(A) Leb_1 (\mathcal H) |J| \nu(B) 
= (\nu(\tau))^2 \mu (\mathcal A) Leb_1 (\mathcal H) \mu(\mathcal B)
$$
Substituting the Riemann sum in \eqref{eq:Rie1} with a Riemann integral, we obtain
\begin{equation}
\label{eq:Rie2}
S_1 \sim 
\left(\frac{\nu(\tau)}{t} \right)^{1/2} \left[ \int_{-K \sqrt{\nu(\tau)}}^{K \sqrt{\nu(\tau)}}
\fg_{\sigma} \left(  - W \sqrt{\nu(\tau)}, \mathfrak z \nu(\tau) \right) d \mathfrak z \right]
\nu(\tau) \mu (\mathcal A) Leb_1 (\mathcal H) \mu(\mathcal B),
\end{equation}
where $\sigma \in GL(2,\mathbb R)$ 
is the covariance matrix of $(\check{\bm \varphi}, \tau)$.
Let us write $\sigma' = \sigma_{11}$.
By the properties of the Gaussian distribution,
$$
\int_{-K \sqrt{\nu(\tau)}}^{K \sqrt{\nu(\tau)}}
\fg_{\sigma} \left(  - W \sqrt{\nu(\tau)}, \mathfrak z \nu(\tau) \right) d\mathfrak z  = 
\frac{1}{\nu(\tau)} \fg_{\sigma'} ( - W \sqrt{\nu(\tau)}) 
(1+o_{K \rightarrow \infty} (1)),
$$
and 
$$
 \fg_{\sigma'} ( - W \sqrt{\nu(\tau)}) =  \sqrt{\nu(\tau)}\fg_{\frac{1}{\nu(\tau)} \sigma'} ( W ) 
$$
Substituting the last two displays into \eqref{eq:Rie2} gives 
\begin{eqnarray*}
S_1 &\sim& \frac{1}{ t^{1/2}} \fg_{\frac{1}{\nu(\tau)} \sigma'} ( W )
\mu (\mathcal A) Leb_2 (\mathcal H) \mu(\mathcal B)
(1+o_{K \rightarrow \infty} (1)).\\
\end{eqnarray*}
This implies
$$
\Xi_t(\mathcal C) \sim \frac{1}{ t^{1/2}} \fg_{\frac{1}{\nu(\tau)} \sigma'} ( W )
\mu (\mathcal A) Leb_2 (\mathcal H) \mu(\mathcal B).
$$
Noting that $H$=$\hat H \;(\bmod\; \mathcal{V})$, we obtain the
MLCLT with variance $\Sigma = \frac{1}{\nu(\tau)} \sigma'$.

\subsubsection{Cases (A) and (B)} 

We  follow the strategy of Section \ref{sec:A1}. 
The main difference
is that now \eqref{eq:Rie1} is replaced by
\begin{equation}
\label{eq:Rie3}
\sum_{n: |n-t/\nu(\tau)| < K \sqrt t} \frac{1}{\nu(\tau) n} 
\fg^*
\int_{s \in I} (\nu \times u_M \times \nu) (C(s) + \varkappa_n) ds
\end{equation}
where 
$M=M(\check{\bm \varphi}, \tau)$, 
$r=r(\check{\bm \varphi}, \tau),$
$\varkappa_n \in \mathbb R^2 /M$ is defined by
$$
\varkappa_n = (Z(t), t) - n r  \;(\bmod\; M)
$$
and $C(s) + \varkappa$ is defined as
$$\{ (x,z+\check \varkappa,y): (x,z,y) \in C(s)\},$$
with some $\check \varkappa \in \mathbb R^2$, $\check \varkappa + M= \varkappa$ (since $u$ is the Haar measure, 
\eqref{eq:Rie3} does not depend on the choice of the representative).

Recall the definition of the linearized group $\hat{\mathcal V}$. 
Writing the sum in \eqref{eq:Rie3} as
$$
\sum_{j=0}^{2 K \sqrt t /N -1} \left(\sum_{m=-K \sqrt t + jN}^{-K \sqrt t + (j+1)N -1} \dots \right),
$$
and using Lemma \ref{lem:Weyl}, we conclude that
\begin{equation}
 \label{eq:caseacsplit}
S_1 \sim
\sum_{n: |n-t/\nu(\tau)| < K \sqrt t} \frac{1}{\nu(\tau) n} 
\fg^*
\int_{s \in I} (\nu \times u_{\hat{\mathcal V}} \times \nu) (C(s) +(W(t),t)) ds
\end{equation}
In cases {\bf (A)} and {\bf (B)}, $\hat{\mathcal V} = \mathcal V \times \mathbb R$ and 
$W(t) \in \mathcal V$. Consequently,
\begin{equation}
 \label{eq:centering}
(W(t),t) \in \hat{\mathcal V}.
\end{equation}
We need to compute
\begin{equation}
\label{eq:hommeasure}
  \int_{s \in I} (\nu \times u_{\hat{\mathcal V}} \times \nu) (C(s) +(W(t),t)) ds = 
\int_{s \in I} (\nu \times u_{\hat{\mathcal V}} \times \nu) (C(s)) ds.
\end{equation}
Integrating with respect to $z_2$ 
we conclude that \eqref{eq:hommeasure} is equal to 
$$
|I| \nu(A) u_{\mathcal V} (\mathcal H) |J| \nu(B).
$$
The rest of the proof is identical to Section \ref{sec:A1}.

\subsubsection{ Case (C)}
We need to adjust the above proof to cover case {\bf (C)} as the measure $u_{\hat{\mathcal V}}$
is not a product in the coordinates $z_1$, $z_2$. Since the proof is similar, we just list the required modifications.
First, 
$\hat H$ is replaced by 
$$
\int_0^{s} \bm \varphi(x,s') ds' + h(x) - \frac{1}{\alpha} (s + h_{\tau}(x)).
$$
Then we need to replace $S_{\psi}(N_{t+s},x)$ by 
$$
S_{\psi}(N_{t+s},x) - \frac{1}{\alpha} \left[ 
S_{\hat \tau}(N_{t+s},x) -t \right]
$$
in formulas \eqref{XiC1} and \eqref{XiC2}.
Also, we replace $C(s)$ by 
the set of points $(x, z, y)$, such that  $x, y \in X$, $z = (z_1,z_2)$ and
$$x \in A; \quad y \in B; \quad z_2 \in -J + s+h_{\tau} (x) - h_{\tau} (y);\quad
z_1 \in \mathcal H + \frac{1}{\alpha} z_2.$$
With this modification, we repeat the previous proof up to the derivation of the formula \eqref{eq:caseacsplit}.
Note that \eqref{eq:centering} and \eqref{eq:hommeasure} hold as well since 
$W(t) \in \frac{\beta}{\alpha} \mathbb Z + t \frac{1}{\alpha}$.
Now we can easily compute \eqref{eq:hommeasure}. Indeed, for any $s,x,y$ fixed, we have
$$
u_{\hat{\mathcal V}} \left( (z_1,z_2): z_2 \in -J + s+ h_{\tau} (x) - h_{\tau} (y), 
z_1 \in \mathcal H + \frac{1}{\alpha} z_2 \right) = u_{\mathcal V} (\mathcal H) |J|.
$$
and we can complete the proof as before.

\subsection{Cases (D) and (E)}

\begin{proof}[Proof of Proposition \ref{prop:pos}]
The proof is similar to the one in Section \ref{sec:casesac}. That is, we define $N_{t+s}$ by
\eqref{DefNu} as before and analogously to
\eqref{XiC2}, we have 
\begin{align}
\Xi_t(\mathcal A \times \{ la\} \times \mathcal B) \nonumber & \\
= \frac{1}{\nu(\tau)}\int_{s \in I} \nu \bigg(  x \in  A: \nonumber &\\
 & S_{\psi}(N_{t+s},x) - W(t) =la , \label{XiD2} \\
 & S_{\hat \tau}(N_{t+s},x) -t \in - J +s + h_{\tau}(x) - h_{\tau} (T^{T_t(x,s)} (x)) \nonumber \\
 & T^{N_{t+s}(x,s)} (x) \in B \bigg) ds. \nonumber
\end{align}
The main difference with Section \ref{sec:casesac} is that the set of $z_2$'s such that 
$(W(t) + la, z_2) \in \hat{\mathcal V}$ is discrete (and not an interval). We will index this set by $k$.
Fix some $s \in I$. 
Given $x \in A$ and $k$ with $|k| \leq \frac{\|\tau  \|_{\infty} + 2\|h_{\tau}\|_{\infty} +1}{d}$, let 
$$
B_{k,x} = \{ y \in B: \rho + kd + h_{\tau}(x) - h_{\tau}(y) \in J\}
$$
and
$$
C_{n,k}(s) = \{ x\in A : S_{\psi} (n,x) - W(t) = la, s+t-S_{\hat \tau }(n,x) = \rho + kd, T^{n} x \in B_{k,x} \}.
$$
Observe that $x \in C_{n,k}(s)$ implies 
$$0 \leq \rho + kd + h_{\tau}(x) - h_{\tau}(T^n x)  = s+t - S_{ \tau }(n,x) < \tau (T^n x),$$ 
and consequently
$n = N_{t+s}(x,s)$. 
It follows that 
\begin{equation}
\label{eq:disjointcnk}
 C_{n,k}(s) \cap C_{n',k'}(s) = \emptyset \text{ unless }n=n' \text{ and } k=k'.
\end{equation}
Let $C^{(k)} =\{ (x,la, s-\rho- kd,y) : x \in A, y \in B_{k,x} \}$. Then,
\begin{equation}
\label{Ckn}
C_{n,k}(s) = \{ x: (x,S_{\psi} (n,x) - W(t), S_{\hat \tau }(n,x) - t, T^n(x)) \in C^{(k)}\}. 
\end{equation}
Note that by the definition of $\rho$, $(W(t) + la, s+t - \rho -kd) \in \hat{\mathcal V}$. Thus we can use 
{\bf (H6)} similarly to the proof of Theorem \ref{thm1}, to deduce
$$
\nu (x: x \in C_{N_{t+s},k}(s)) \sim \frac{ \fg_{\Sigma}(W)}{\nu(\tau) \sqrt t} 
(\nu \times \nu) (x\in A, y \in B_{k,x}) \vol(\mathbb R^2 / \hat{\mathcal V}). 
$$
Next, we have by \eqref{eq:disjointcnk} that 
$$
\nu (x: x \in \cup_k C_{N_{t+s},k}(s)) = \sum_k \nu (x: x \in C_{N_{t+s},k}(s)).
$$
We obtain the result by integrating with respect to $s \in I$ (using dominated convergence) and using
that $\vol(\mathbb R^2 / \hat{\mathcal V}) = ad$ and 
$\mu = \nu \otimes Leb_1/\nu(\tau)$.
\end{proof}

\begin{proof}[Proof of Proposition \ref{prop:E}]
Case {\bf (E)} can be reduced to Case {\bf (D)}  similarly to the reduction of Case {\bf (C)} 
to Case {\bf (B)}. We only list the adjustments needed to the proof of Proposition \ref{prop:pos}.
First, 
we  replace $S_{\psi}(N_{t+s},x) - W(t)$ by
$$
S_{\psi}(N_{t+s},x) - W(t) - \frac{c'}{d'} [S_{\hat \tau}(N_{t+s},x) -t]
$$
in  \eqref{XiD2} and in the definition of $C_{n,k}(s)$. Then, we replace 
$C^{(k)}$ by
$$
\{ (x,la + \frac{c'}{d'}(s-\rho- kd), s-\rho- kd,y) : x \in A, y \in B_{k,x} \}.
$$
Note that \eqref{Ckn} is unchanged. Finally,  instead of $(W(t) + la, s+t - \rho -kd) \in \hat{\mathcal V}$
we have
$$
\left( W(t) + la + \frac{c'}{d'}(s-\rho- kd), s+t - \rho -kd \right) \in \hat{\mathcal V}.
$$
The last inclusion holds by conditions  (2) and (4) 
of Proposition \ref{prop:E}. Now we can apply 
{\bf (H6)} as before.
\end{proof}

\section{Expanding Young towers}
\label{secpoly}

\subsection{Setup and results}
\label{sec:towerpolydef}

Let $(\Delta, \tilde \nu)$ be a probability space with a partition 
$(\Delta_{k,l})_{ k \in I,
  l \leq r_k}$
into positive measure subsets, where $I$ is either finite or countable and 
$r_k = r(\Delta_{0,k})$ is a positive
integer. Let $F: \Delta \rightarrow \Delta$ be a map that satisfies the following
\begin{enumerate}
\item[(A1)] for every $i \in I$ and $0 \leq j < r_i -1$, $F$ is a measure preserving isomorphism
between $\Delta_{i,j}$ and $\Delta_{i,j+1}.$
\item[(A2)] for every $i \in I$, $F$ is an isomorphism between $\Delta_{i, r_i -1}$
and $$X:= \Delta_0 := \cup_{i \in I} \Delta_{i,0}.$$
\item[(A3)] Let $r(x) = r(\Delta_{0,k})$ if $x \in \Delta_{0,k}$ and 
$T:X \rightarrow X$ be the first return map to the base, i.e. $T(x) = F^{r(x)} (x)$.
Let $s(x,y)$, the separation time of $x,y \in X$, be defined as the smallest integer $n$
such that $T^n x \in \Delta_{0,i}$, $T^n y \in \Delta_{0,j}$ with $i \neq j$. 
As $T: \Delta_{0,i} \rightarrow X$ is an isomorphism, it has an inverse. Denote by
$g$ the jacobian of this inverse (w.r.t. the measure $\tilde \nu$). Then there are constants
$\beta < 1$ and $C>0$ such that for every $x,y \in  \Delta_{0,i} $, 
$|\log g(x) - \log g(y) | \leq C \beta^{s(x,y)}.$
\item[(A4)] Extend $s$ to $\Delta$ by setting $s(x,y) = 0$ if $x,y$ do not belong to the same
$\Delta_{i,j}$ and $s(x,y) = s(F^{-j}x,F^{-j}y) + 1$ if $x,y \in \Delta_{i,j}$.
$(\Delta, \tilde \nu, F)$ is exact (hence ergodic and mixing) with respect to the
metric 
\begin{equation}
\label{DefSymbMetr}
d(x,y) = \beta^{s(x,y)}.
\end{equation}
\end{enumerate}

\medskip\noindent
See \cite{Y99} for the introduction and several examples of such maps.

The measure defined by $\nu(A) = \tilde \nu (A) / \tilde \nu (X)$ for $A \subset X$ is invariant
for $T$.
Note that $\nu(A) = \nu(r) \tilde \nu (A)$.
We assume that
\begin{equation}
\label{TailPoly}
\nu (x: r(x) > n)  =  \nu(r) \tilde \nu (x \in X: r(x) > n) = O(n ^{- \beta})
\end{equation}
with $\beta >2$.

We consider the space of dynamically H\"older functions on $\Delta$: 
$$
C_{\varkappa}(\Delta, \mathbb R^d) = \{f: \Delta \rightarrow \mathbb R^d \text{ bounded and } \exists C: |f(x) -f(y)| \leq C 
\varkappa^{s(x,y)} \},
$$
where $\varkappa < 1$ is fixed and $s(x,y)$ is the separation time of $x$ and $y$ 
(there is no
major difference between H\"older and Lipschitz terminologies as one can increase 
$\beta<1$
in the definition of the metric \eqref{DefSymbMetr}).
We will use 
the notation $C_{\varkappa}(X, \mathbb C) $ for the space of functions with domain $X$ and range $\mathbb{C}$, 
defined analogously to $C_{\varkappa}(\Delta, \mathbb{R}^d).$
The corresponding norm is
$$
\|  f\|_{\varkappa} = \inf \{ C: | f(x) -  f(y)| \leq C \varkappa^{s(x,y)} \forall x,y  \} + \|  f \|_{\infty} .
$$
To given $ f \in C_{\varkappa}(\Delta, \mathbb R^d)$, we associate the function
$f_X : X \rightarrow \mathbb R^d$ where
\begin{equation}
\label{IndSum}  
\displaystyle f_X(x) = \sum_{i=0}^{r(x)-1} f(F^i x).
\end{equation}

\begin{theorem}
\label{propYoungpoly}
 Let $T: X \rightarrow X$ be as above.
 Assume that $(\check{\bm \varphi}, \tau) = f_X$ for some $f = 
(\check{\tilde{\bm \varphi}}, \tilde \tau)
\in C_{\varkappa}(\Delta, \mathbb R^2)$, 
where $f_X$ is non-degenerate and $\tau$ is bounded away from zero.
 Then  {\bf (H1)}, {\bf (H2)}, {\bf (H6)} and {\bf (H7)} hold.
\end{theorem}

Before proving Theorem \ref{propYoungpoly}, we make several remarks and derive 
several corollaries from that theorem.

Note that $\beta = \alpha + 1$ in the notation of \cite{Y99}. 
By the results of \cite{Y99}, $C_{\varkappa}(\Delta, \mathbb R)$ 
observables decorrelate at the speed $O(n^{-\beta +1})$
and satisfy the CLT as long as $\beta >2$. 

We would like to conclude the MLCLT for suspensions over
some maps which can be modeled
by a tower with polynomial tails. 
Let $\mathbb  F :\mathbb M \rightarrow \mathbb M$ be a map on a compact Riemannian manifold
$\mathbb M$
with invariant measure $\lambda$ that satisfies assumptions 1-4 in Section 1.3.1. of \cite{G05} 
(in \cite{G05}, $\mathbb M,\mathbb F,\lambda$ are denoted by $X,T,\nu$, respectively). 
We also
assume that $\lambda$ is the unique SRB measure for $\mathbb F.$
Let $\upsilon$ be a H\"older roof function on $\mathbb M$ and $\Psi$ be the corresponding suspension 
(semi-)flow on the phase space
$\aleph$.
Let $\bm \chi: \aleph \rightarrow \mathbb R$ be a zero mean continuous observable such that
$\check{\bm \chi}$ is H\"older.
Then, as explained in Section 1.3.1. of \cite{G05}, we can construct a tower $(\Delta, F, \tilde \nu)$,
with $X = \Delta_0 \subset \mathbb M$
satisfying assumptions (A1)--(A4) above
and a H\"older mapping $\rho: \Delta \rightarrow \mathbb M$ so that 
$\rho \circ F=\mathbb F \circ \rho$ and $\rho_* \tilde \nu = \lambda$.
Define $\tilde \tau: \Delta \rightarrow \mathbb R_+ $ 
by $\tilde \tau (x) = \upsilon (\rho(x))$. Let $\tilde \Omega$ be the phase space of the suspension (semi-)flow over 
 $(\Delta, F, \tilde \nu)$ with roof function $\tilde \tau$ and let 
 $\tilde{\bm \varphi}: \tilde \Omega \rightarrow \mathbb R$ be defined by
 $\tilde{\bm \varphi} (x,s) = \bm \chi(\rho(x), s)$. As before we let $(X,T,\nu)$
be the first return to the base of the tower. 
Let 
$\Omega$ be the phase space of the 
suspension (semi-)flow over $(X,T,\nu)$ with roof function $\tau = \tilde{\tau}_X$.
Let ${\bm \varphi}: \Omega \rightarrow \mathbb R$ be defined by
 ${\bm \varphi} (x,s) = \tilde{ \bm \varphi} (x,s)$
(here, $s \in [0, \tau (x))$, thus $\tilde{ \bm \varphi} (x,s)$ is to be interpreted with the usual identification
$(x, \tilde \tau(x)) = (F x, 0) \in \tilde \Omega$).
Note that $(\aleph, \kappa =  \frac{1}{\lambda(\upsilon)} \lambda \otimes Leb, \Psi^t)$
is a factor of $(\Omega, \mu=\frac{1}{\nu(\tau)}\nu \otimes Leb, \Phi^t)$. Indeed,
the mapping $\iota : \Omega \rightarrow \aleph$, 
$\iota (x,s) = (x,s)$ is a homomorphism (mind the identification
$(x, \upsilon(x)) = (\mathbb F(x), 0) \in \aleph$) which is is in general not invertible.
We can lift up the test function $\mathfrak X: \aleph \rightarrow \mathbb R$, 
to ${\mathfrak V}: \Omega \rightarrow \mathbb R$ by ${\mathfrak V} = 
\mathfrak X \circ \iota$ (similarly, let 
${\mathfrak W} =
\mathfrak Y \circ \iota$).
Now we have by definition
\begin{align}
&\int_{\aleph} \mathfrak X (x,s) \mathfrak Y( \Psi^t (x,s))
\mathfrak Z \left( \int_0^t \bm \chi(\Psi^{s'}(x,s))ds'   - W(t) \right) d\kappa(x,s) \nonumber \\
=&\int_{\Omega} {\mathfrak V} (x, s) {\mathfrak W}( \Phi^t (x,s))
\mathfrak Z \left( \int_0^t {\bm \varphi}(\Phi^{s'}(x,s))ds'   - W(t) \right) d\mu(x,s). 
\label{eq:factor2}
\end{align}

This identity, combined with Theorem \ref{propYoungpoly} readily gives

\begin{proposition}
\label{corpoly1}
In the setup of the previous paragraph, let us assume that
$(\Delta, F, \tilde \nu)$ satisfies \eqref{TailPoly} with $\beta >2$.
Furthermore, assume that $(\check{\bm \varphi}, \tau)$ is minimal and its linearized
group falls into cases {\bf (A)}, {\bf (B)} or {\bf (C)}. 
Then 
the conclusion of Theorem \ref{thm1} holds for $(\Phi, \bm \varphi)$
with $h(x) = 0$, $h_{\tau}(x) = 0$.
\end{proposition}

Sometimes the flow $\Psi^t$ does not have a canonical coding as a suspension of a Young tower map.
For example, $\Psi^t$ can be given by an ODE. In this case
it is important to reformulate 
Proposition \ref{corpoly1} as an MLCLT for $(\Psi, \bm \chi)$. 
However, if we want to do so, the following two difficulties arise:

\begin{enumerate}
\item[(D1)] $H: \Omega \rightarrow \mathbb R$, given by Theorem \ref{thm1} may not be the lift-up of a function on 
$\aleph$.
\item[(D2)] for given bounded and continuous test functions $\mathfrak X,\mathfrak Y: 
\aleph \rightarrow \mathbb R$, the corresponding lift-ups $\mathfrak V, \mathfrak W: \Omega \rightarrow \mathbb R$
may not be supported on $\{ (x,s) \in \Omega: s \leq M\}$ for some finite $M$.
\end{enumerate}

It is easy to overcome (D1) in Case {\bf (A)} since $H=0$. In Cases {\bf (B)} and {\bf (C)}
we do not know how to overcome (D1) in general. However, at least we can distinguish between cases {\bf (A)}, {\bf (B)} and {\bf (C)}
by only looking at the manifold itself (rather than at the tower).

\begin{proposition}
\label{proppoly2}
$\hat{\mathcal V}_{(\check{\bm \chi}, \upsilon)}$ and 
$\hat{\mathcal V}_{(\check{\bm \varphi}, \tau)} $ fall into the same case {\bf (A)} - {\bf (E)}.
\end{proposition}

Next, we introduce the notion of minimal group of an observable for an abstract dynamical system.


\begin{definition}
\label{defgroup}
Given any dynamical system $\mathcal T$ on a measurable space $(\mathcal X, \zeta)$ and an observable 
$u: \mathcal X \rightarrow \mathbb R^d$, we define
$S(u)$ as the minimal closed group, a translate of which supports 
the values of  $u$.
Let us define the minimal closed group $M(u)$ 
by $M(u) = \cap_{v \sim u} S(v)$, where $v \sim u$ means that 
$u + h - h \circ \mathcal T$ holds with some $h = h_u: \mathcal X \rightarrow \mathbb R^d$ measurable.
Let $r(u) \in \mathbb R^d / M(u)$ be the translate, i.e. $range(v) \subset M(u) + r(u)$ with some $v \sim u$.
A function $u$ is called minimal if $M(u) = S(u)$ and is called non-arithmetic if $M(u) = \mathbb R^d$.
\end{definition}

The notation $M(u), r(u)$ 
was used in Sections \ref{sec2} and~\ref{sec3} where $M$ was the symmetry group 
of the local distribution in the MCLLT and $r$ was the shift.
 We will see in the proof of Theorem \ref{propYoungpoly}, that for Young towers
the local limit theorem holds with $M,r$ given by Definition \ref{defgroup}.

Our next result lifts the MLCLT to $\Psi$ in a special case.

\begin{proposition}
\label{propunbdH}
In the setup of Proposition \ref{corpoly1}, let us also assume that 
$(\check{\bm \chi}, \upsilon)$ is non-arithmetic. Then for any continuous $\mathfrak X, \mathfrak Y: \aleph \rightarrow \mathbb R$,
any continuous and compactly supported $\mathfrak Z: \mathbb R \rightarrow \mathbb R$,
and any $W(t)$ with $W(t) / \sqrt t \rightarrow W$,
\begin{align}
&\lim_{t \rightarrow \infty} t^{1/2} \int_{\aleph} \mathfrak X (x,s) \mathfrak Y( \Psi^t (x,s))
\mathfrak Z \left( \int_0^t \bm \chi(\Psi^{s'}(x,s))ds'   - W(t) \right) d\kappa(x,s) \nonumber \\
&= \mathfrak g_{\Sigma} (W) 
\int \mathfrak X d \kappa \int \mathfrak Y d \kappa \int \mathfrak Z d Leb .
\nonumber
\end{align}
\end{proposition}



\subsection{Proof of Proposition \ref{proppoly2}}

Let $f_{\mathbb M}: \mathbb M \rightarrow \mathbb R^d$ be an observable.
Let $f_{\Delta}: \Delta \rightarrow \mathbb R^d$ be its lift-up, i.e. 
$f_{\Delta} (x) = f_{\mathbb M} ( \rho(x)) $ and $f_{X}: X \rightarrow \mathbb R^d$ be the corresponding observable on $X$,
i.e. $f_X(x) = \sum_{i=0}^{r(x) -1} f_{\Delta} (F^ix)$.

\begin{lemma}
\label{lem:Gouezel1.4}
$M(f_{\Delta})$ is a finite index subgroup of $M( f_{\mathbb M})$ and $r(f_{\mathbb M}) = \iota' (r (f_{\Delta}))$, where $\iota'$ is the natural 
surjective homomorphism from $\mathbb R^d / M(f_{\Delta})$ to $\mathbb R^d / M(f_{\mathbb M})$.
\end{lemma}

\begin{proof}
The one dimensional case is proved in  \cite[Theorem 1.4]{G05}. 
 
We can assume $r(f_{\Delta})=0$ by possibly adding a constant to both $f_{\mathbb M}$ and $f_{\Delta}$. 
Clearly $M(f_{\Delta})$ is a subgroup of $M( f_{\mathbb M})$, as we can lift up any cohomological
equation from $\mathbb M$ to $\Delta$.
$M(f_{\Delta})$ is isomorphic to $\mathbb R^{\tilde d_1} \times \tilde{ \mathcal L}$, where $\tilde{ \mathcal L}$ is a $\tilde d_2$ dimensional lattice and 
$\tilde d_3 := d - \tilde d_1 - \tilde d_2 \geq 0$. 
Similarly, $M(f_{\mathbb M})$ is isomorphic to $\mathbb R^{d_1} \times \mathcal L$, where $\mathcal L$ is a $d_2$ dimensional lattice and 
$d_3 = d - d_1 - d_2 \geq 0$. Furthermore, by the subgroup property, $\tilde d_1 \leq d_1$ and $d_3 \leq \tilde d_3$.
It remains to show that $\tilde d_1 = d_1$ and $\tilde d_3 = d_3$. 
Replacing $f_{\mathbb M}$ and $f_{\Delta}$ by
$A f_{\mathbb M}$ and $A f_{\Delta}$
with some invertible matrix $A$, we can assume that $M(f_{\Delta}) = \mathbb R^{\tilde d_1} \times \mathbb Z^{\tilde d_2}$. 
Next we observe that for an $\mathbb R^d$ - valued function $u$, 
\begin{equation}
\label{eq:proj}
\text{Cl}(\pi_{V}) M(u)=M(\pi_V u), 
\end{equation}
where Cl means closure and
$\pi_{V}$ is the orthogonal projection to the (coordinate) subspace $V$. Similarly, let $\pi_k$ be the projection to the $k$th coordinate
subspace.

If $\tilde d_3 > d_3$, applying \eqref{eq:proj} to the $(\tilde d_1 + \tilde d_2 +1)$st coordinate direction, we obtain 
$M(\pi_{\tilde d_1 + \tilde d_2 +1}f_{\Delta}) = \mathbb Z$ while $M( \pi_{\tilde d_1 + \tilde d_2 +1} f_{\mathbb M}) = 0$, which is a contradiction
with the one dimensional case. Thus $\tilde d_3 = d_3$. 
Similarly, if $d_1 > \tilde d_1$, then there exists some $d'$ with $\tilde d_1 < d' \leq \tilde d_1 + \tilde d_2$ such that
$\mathbb R = \text{Cl}(\pi_{d'} M(f_{\mathbb M}))=  M(\pi_{d'} f_{\mathbb M})$. On the other hand, $ M(\pi_{d'} f_{\Delta}) = \mathbb Z$, which is again a contradiction with
the one dimensional case.
\end{proof}

\begin{lemma}
\label{lem:suspgroup}
The linearized groups of $f: \Delta \rightarrow \mathbb R^d$ and $f_X: X \rightarrow \mathbb R^d$ 
are the same, i.e. $\hat{\mathcal V}_{f} = \hat{\mathcal  V}_{f_X}$.
\end{lemma}

\begin{proof}
Assume that $f = g + h - h \circ F$, where $g: \Delta \rightarrow \hat{\mathcal V}_{f}$. Then 
$f_X = g_X + h - h\circ T$, where $g_X: X \rightarrow \hat{\mathcal V}_{f}$. 
Thus $\hat{\mathcal V}_{f_X} \subseteq \hat{\mathcal V}_{f} $.
Next, assume that for some $f: \Delta \rightarrow \mathbb R^d$, we have 
$f_X = \underline g + h - h\circ T$, with $\underline g: X \rightarrow \hat{\mathcal V}_{f_X}$.
Let us define $\tilde f: \Delta \rightarrow \mathbb R^d$ by
$\tilde f (x,l) = \underline g(x) 1_{\{l = 0\}} + h_l(x) - h_{l+1}(x)$, where 
$h_0(x) = h(x)$, $h_{r(x)}(x) = h(Tx)$ and $h_l(x) = 0$ if $l \notin \{ 0, r(x)\}$. 
Thus $\hat{\mathcal V}_{\tilde f} \subseteq \hat{\mathcal V}_{f_X} $
By construction, $(\tilde f)_X = f_X$ and thus $f$ and $\tilde f$ are cohomologous, 
$\hat{\mathcal V}_{\tilde f} = \hat{\mathcal V}_{f} $.
\end{proof}

\begin{proof}[Proof of Proposition \ref{proppoly2}]
Observe that $\hat{\mathcal V}_{(\check{\bm \chi}, \upsilon)}$ is a finite index subgroup of
$\hat{\mathcal V}_{(\tilde{\check{\bm \varphi}}, \tilde \tau)}$ by Lemma \ref{lem:Gouezel1.4}.
Thus, by Lemma \ref{lem:suspgroup}, $\hat{\mathcal V}_{(\check{\bm \chi}, \upsilon)}$
is a finite index subgroup of $\hat{\mathcal V}_{(\check{\bm \varphi}, \tau)} $. In particular, 
$\hat{\mathcal V}_{(\check{\bm \chi}, \upsilon)}$ and 
$\hat{\mathcal V}_{(\check{\bm \varphi}, \tau)} $ fall into the same cases {\bf (A)}--{\bf (E)}.
\end{proof}

\subsection{Proof of Theorem \ref{propYoungpoly}}

  The main ingredients of the proof are

  (i) MLCLT with an error estimate and
  
   (ii) moderate and large deviation estimates. 

We start with establishing MLCLT with rates in parts (A) and (B) of Lemma~\ref{lemtowerLLT}.
Lemma \ref{lemtowerLLT}(C) provides a useful generalization of Lemma \ref{lemtowerLLT}(B).  
The required moderate and large deviation bounds are contained
in Lemmas~\ref{lemmaMel}--\ref{lemmaLSLD}.

\begin{lemma}
\label{lemtowerLLT}
Consider the setup of Theorem \ref{propYoungpoly} with either $d=1$ and $f_X = \tau$ or
$d=2$ and $f_X = (\check{\bm \varphi}, \tau) $ with $\nu(\check{\bm \varphi}) = 0$. Then\\

\noindent
(A) $ (T,f_X)$ satisfies the MLCLT.
$$(B) \quad \nu (x: S_{f_X}(n,x) - n \nu(f_X) \in B(v, R)) \leq C \left( n^{-\frac{d}{2}} \fg_{\sigma}(v/\sqrt n) 
+ n^{-\frac{d+\beta}{2} + 1} + n^{-\frac{d+1}{2}}\right).$$
$$(C)\quad \nu (x: S_{f_X}(n,x) - n \nu(f_X) \in B(v, R), T^n x \in \Delta_{0,l}) 
\quad\quad\quad\quad\quad\quad\quad\quad\quad\quad
$$
$$\leq C \nu(\Delta_{0,l}) \left( n^{-\frac{d}{2}} \fg_{\sigma}(v/\sqrt n)
+ n^{-\frac{d+\beta}{2} + 1} + n^{-\frac{d+1}{2}}\right) .$$
\end{lemma}

\begin{proof}
Lemma \ref{lemtowerLLT}
can be prove
by the Fourier method, 
cf. similar results in \cite{AD01,G05, GH88, R83, SzV04}.
We briefly sketch the proof here only highlighting the differences from the analogous arguments in the above list. 

Let us assume first that $f_X$ is minimal in the sense of Definition \ref{defgroup}.
The proof is based on $P$, the Perron-Frobenius operator associated to $T$ acting on 
$C_{\varkappa}(X, \mathbb C)$ by the formula
$$ \nu(f (g\circ T))=\nu(g Pf). $$
and the twisted operators $P_{t}u = P(e^{i \langle t, f_X\rangle} u)$ for 
$t \in \mathbb R^d$. These operators satisfy the Lasotha-Yorke (a.k.a. Doeblin-Fortet) inequality:
\begin{equation}
 \label{LY}
\| P_t^n u \|_{\varkappa'} \leq C(1+|t|) \| u\|_{\varkappa'} + C \| u \|_{L_1}
\end{equation}
(with some $\varkappa ' \in (\varkappa, 1)$).
In the one dimensional case, \eqref{LY} is included in Lemma 4.1(2) of \cite{G05} and its proof is sketched based
on Proposition 2.1 of \cite{AD01}. Note that Proposition 2.1 of \cite{AD01} is valid in higher dimensions, so the 
adjustments
described in \cite{G05} give \eqref{LY}. By classical results of \cite{I-TM50}, $P$ is quasicompact with a 
simple eigenvalue at $1$ (eigenfunction identically $1$ as we took the Jacobian w.r.t. $\nu$ in the definition)
and finitely many eigenvalues with modulus in $(\rho, 1]$ for some $\rho <1$. By perturbation theory, similar picture holds for 
$P_t$ for $t$ in a small neighborhood of $0 \in \mathbb R^d$. We need to understand the asymptotics of 
$\lambda_t$, the eigenvalue of $P_t$ close to $1$, and the other eigenvalues on the unit circle.

If $d=2$, let us fix a vector $s = (s_1, s_2) \in \mathbb R^2$ of unit length, and for a moment, 
only consider $t \in \mathbb R^2$ in the direction of $s$, i.e.  $t = |t| s$. Then 
\cite[Prop. 4.5]{G05} applied to the function $\langle s, f\rangle$ tells us that
$$\lambda_{|t| s} = \nu(e^{i |t| \langle s, f\rangle}) + (c_1 s^2_1 + c_2 s^2_2) t^2 +O\left(|t|^3\right).$$
On the other hand the assumption \eqref{TailPoly} implies that 
$\nu(f_x>R)=O\left(R^{-\beta}\right)$,
and hence
$$\nu(e^{i |t| \langle s, f\rangle}) = 
1 + i \nu(\tau) s_2 |t| - (\tc_1 s^2_1 + \tc_2 s^2_2) t^2 +O(t^{\beta} + t^3)$$
(see e.g.  \cite{W73}). Combining the last two displays we obtain
$$\lambda_{|t| s} = 1 + i \nu(\tau) s_2 |t|- (\hc_1 s^2_1 + \hc_2 s^2_2) t^2 +O(|t|^{\beta} + |t|^3),$$
with some $\hc_1, \hc_2$, which are independent from $s$.
By the proof of the cited proposition, the implied constant in $O(\cdot)$ only depends on 
$ \| \langle s, f\rangle \|_{C_{\varkappa}(X, \mathbb R)},$ hence it is uniform in $s$. We conclude that 
there is a matrix $\bm m$
\begin{equation}
\label{T3Eigen}
\lambda_{t} = 1 + i (0, \nu(\tau)) t- t^T \bm m t +O(|t|^{\beta} + |t|^3).
\end{equation}
In fact, the coefficients of $\bm m$ are given by Green-Kubo formula  
(see the expression for $a$ in \cite[Prop 4.5]{G05}) but we will not need this fact.

If $d=1$, then 
Proposition 4.5 of \cite{G05} directly implies that there is a constant $m>0$ such 
$\lambda_{t} = 1 + i  \nu(\tau) t- t^2 m  +O(|t|^{\beta} + |t|^3)$.

The characterization of other eigenvalues of $P_t$ on $\mathcal S^1$ is 
  again analogous
to similar computations in \cite{AD01, G05}:  $\lambda \in \mathcal S^1$ is an eigenvalue (and 
in this case, $g=g_t$ is an eigenfunction with $|g_t| = 1$) if and only if
$e^{i \langle t, f_X \rangle} g_t = \lambda g_t \circ T$. Then we can finish the proof for $f_X$ minimal
as in the above references.

Now assume $f_X$ is not minimal, i.e., 
$f_X = \varsigma  + h - h \circ T$ for some measurable $\varsigma, h$ and $M=M(f_X) = S(\varsigma) \subsetneq S(f_X)$. 
A priori we only know that $h$ and $\varsigma$ are measurable. In order to prove the MLCLT, we need to show that 
we can choose $h$ and $\varsigma$ so that $h$ is bounded and almost everywhere continuous. In fact, we will show that
it is Lipschitz and consequently $\varsigma$ is Lipschitz as well. Then we can repeat the previous argument
with $f_X$ replaced by $\varsigma$ to conclude the MLCLT.
 
Let $G$ be defined by $M = \widehat{\mathbb R^2 / G}$, where $\widehat{L}$ is the
group of characters of $L$.
By \cite{AD01}, Proposition 3.7, for any $t \in G$, there is some $\lambda \in \mathcal S^1$ and
a Lipschitz function $g_t: X \rightarrow \mathcal S^1$
so that $e^{i \langle t, f_X(x) \rangle} = \lambda g_t(x) / g_t(T(x))$. Next, we show that there is a function
$h_t: X \rightarrow \mathbb R$ which is Lipschitz and satisfies $e^{ih_t} = g_t$. 
Since $g_t$ is Lipschitz, there is some $K$ so that the oscillation of $g_t$ on $K$-cylinders is less than $\sqrt 2$
(we call a cylinder of length $K$ a set of the type
$\Delta^0_{i_1,...,i_K} = \cap_{j=1}^K T^{-j +1} \Delta_{0, i_j}$). Fix a cylinder $\Delta^0_{i_1,...,i_K}$ and pick an arbitrary
element $x_{i_1,...,i_K} \in \Delta^0_{i_1,...,i_K}$. Let $h_t(x_{i_1,...,i_K})$ be the unique real number 
in $[0, 2 \pi)$ that satisfies $e^{ih_t(x_{i_1,...,i_K})} = g_t(x_{i_1,...,i_K})$. By the choice of $K$, for any 
  $y \in \Delta^0_{i_1,...,i_K}$ there is a unique 
  $h_t(y) \in (h_t(x_{i_1,...,i_K}) - \pi, h_t(x_{i_1,...,i_K}) + \pi)$
satisfying $e^{ih_t(y)} = g_t(y)$. By construction, $h_t$ is Lipschitz.
Observe that
$$\langle t, f_X(x) \rangle = \rho +h_t(x) - h_t(T(x)) + \varsigma_t(x),$$ where $\lambda = e^{i \rho}$
and $\varsigma_t: X \rightarrow 2 \pi \mathbb Z$ is some function.
Since $f_X$ is non-degenerate, Proposition 3.9 of \cite{AD01} implies that $G$ is a discrete group.
Now we define $h: X \rightarrow$ span$(G)$ via $\langle t, h\rangle = h_t$ for all $t$ in a fixed generator of $G$.
By construction, $h$ is Lipschitz. Lemma \ref{lemtowerLLT} (A) follows.

Next, we prove part (B).
Consider the non-negative function
$$ h_1(z) = \frac{1-\cos(\varepsilon z)}{\pi\varepsilon^2 z^2} $$
(with some small $\varepsilon >0$). Its Fourier transform equals
$$
\hat h_1 (t) = \int e^{itz} h_1(z) dz = 
1_{|t|< \varepsilon} \left( \frac{1}{\varepsilon} - \frac{|t|}{\varepsilon^2} \right).
$$
Next, consider $h_2(z_1,z_2) = h_1(z_1) h_1(z_2)$ and its Fourier transform
$\hat h_2(t_1, t_2) = \hat h_1(t_1) \hat h_1(t_2)$.
We will prove that
\begin{eqnarray}
&&\int h_d (S_{f_X}(n,x)-n \nu(f_X) -v ) \nu(dx) \nonumber \\
&&\leq C \left( n^{-\frac{d}{2}} \fg_{\sigma}(v/\sqrt n) 
+  n^{-\frac{d+\beta}{2} + 1} + n^{-\frac{d+1}{2}} \right). \label{eq:Polya}
\end{eqnarray}
This will imply Lemma \ref{lemtowerLLT} (B) 
(with some different $C$) as $h_d(z) \geq c >0$ for $\|z \| < \varepsilon/2$ and we can cover a ball of radius $R$
with balls of radius $\varepsilon /2$.

The proof of \eqref{eq:Polya} is standard. Namely, we rewrite the LHS as
\begin{equation}
\label{eq:PolyaFT}
\left(\frac{1}{2\pi}\right)^d \int_{[-\eps, \eps]^d} \hat h_d(t) 
\nu\left(P_t^n 1\right) e^{-i t(n\nu(f_X)+v)} dt 
\end{equation}
and estimate $|\nu(P_t^n 1)|$ taking the second order Taylor expansion of $P_t$ at zero.
The main contribution comes from the leading eigenvalue which is controlled by 
\eqref{T3Eigen}. A similar computation can be found in Sections A.2-A.4 of \cite{P09}.
The setting of \cite{P09} is different since only lattice distributions are considered where
but the lattice assumption is only used to ensure that the integration in \eqref{eq:PolyaFT} is over a compact
set. In our case the compactness comes from the fact that $\hat h_d$ has compact support.
Thus the proof of \eqref{eq:Polya} is similar to \cite{P09}, so we leave it to the reader.

The proof of part (C) is the same as the proof of
part (B) except that  in \eqref{eq:PolyaFT} 
$\nu(P_t^n 1)$ 
has to be replaced by $\nu(1_{\Delta_{0,l}} P_t^n 1)$
providing an additional improvement by the factor $\nu(\Delta_{0,l}).$ 
We refer the reader to the penultimate 
formula on page 834 in \cite{P09} for a similar argument.
\end{proof}

Next, we proceed to the desired moderate and large deviation estimates. 
We first prove global bounds and then
 derive local bounds from the global ones.

\begin{lemma}
\label{lemmaMel}
 ({\em Global moderate deviations})
Consider the setup of Theorem \ref{propYoungpoly} with $f_X = \check{\bm \varphi}$
or $f_X = \tau$. 
For any $\varepsilon >0$ there
is some constant $C= C_{\varepsilon}$ such that
for any $\xi \in (1/2,1]$
\begin{equation*}
\nu(x \in X: | S_{f_X}(n,x) - n \nu(f_X) | > n^{\xi} ) \leq 
C_{\xi, \varepsilon} n^{-(\beta -1)(2 \xi -1) + \varepsilon}.
\end{equation*}
\end{lemma}

\begin{proof}
Let us write $\tilde S_f (n,x) = \sum_{i=0}^{n-1} f(F^i(x))$ for $x \in \Delta$ where
$f \in C_{\varkappa}(\Delta, \mathbb R^d)$ is related to $f_X$ via \eqref{IndSum}.
\cite[Thm 1.3]{M09} states that for $\xi \in (1/2,1]$, 
\begin{equation}
\label{eq:Mel1}
\tilde \nu(x \in \Delta : | \tilde S_{f} (n,x) - n \tilde \nu(f) | > n^{\xi} ) \leq C_{\varepsilon} (\ln n)^{\beta -1 } 
n^{-(\beta -1)(2 \xi -1) }. 
\end{equation}

We will deduce Lemma \ref{lemmaMel} from \eqref{eq:Mel1}. 
Let $\chi_{roof}: \Delta \rightarrow \{0,1\}$ be the indicator function of the top floor of $\Delta$. With the 
notation
$$
t_n : X \rightarrow \mathbb Z, \quad t_n(x) := \min \{ m: \tilde S_{\chi_{roof}} (m,x) = n\}
$$
we have 
$S_{f_X} (n,x) = \tilde S_{f}(t_n(x),x)$. Next, we fix some
 $\xi'' < \xi' < \xi$ close to $\xi$ and write $n_{\pm} = \lfloor n \tilde \nu (r) \pm n^{\xi'} \rfloor$.
We claim that for $n$ large enough,
 $$
 \{ x \in X: |S_{f_X} (n,x) - n \nu (f_X) | > n^{\xi}\} \subset \bigcup_{i=1}^4 A_i,
 $$
 where
\begin{align*}
A_1 & = \{ x \in X: |\tilde S_f (n \nu(r), x) - n \nu(f_X)| > n^{\xi}/4 \}\\
A_2 & =  \{ x \in X: \exists k \in [n_-, n_+]: | \tilde S_f (n \nu(r), x) -  \tilde S_f (k, x) | > n^{\xi}/4 \} \\
A_3 & =  \{ x \in X:  | \tilde S_{\chi_{roof}} (n _-, x) -  n_-\tilde \nu (\chi_{roof}) | > n^{\xi''} \} \\
A_4 & =  \{ x \in X:  | \tilde S_{\chi_{roof}} (n _+, x) -  n_+ \tilde \nu (\chi_{roof}) | > n^{\xi''} \}.
\end{align*}
To verify the above claim, observe that
$$\{ x: t_n(x) < n_-\} \subset A_3\quad\text{and}\quad \{ x: t_n(x) > n_+\} \subset A_4. $$
Assuming (as we can) that $\xi'' = \xi''(\xi,\varepsilon)$ is sufficiently close to $\xi$,
 \eqref{eq:Mel1} gives that $\nu( A_1)+  \nu ( A_3) + \nu ( A_4)\ll n^{-(\beta -1)(2 \xi -1) + \varepsilon}$
Since $f$ is bounded, 
$A_2 = \emptyset$ for $n$ sufficiently large. We have finished the proof of Lemma \ref{lemmaMel}.
\end{proof}

\begin{lemma}
\label{lemmaLLD}
 ({\em Local moderate deviations}) For any $\varepsilon, \varepsilon' >0$ fixed, there exists a constant $C= C_{\varepsilon,\varepsilon'}$ such that 
 
 \noindent
($d=1$) For any $L$ with $|L| \geq n ^{1/2 + \varepsilon}$
$$\nu(x \in X: S_{\tau}(n,x) - n \nu(\tau) \in [L, L+1] ) \leq 
C \frac{n^{\beta - 3/2+ \varepsilon'}}{(\min \{ L,n\})^{2(\beta -1)}}.
$$
($d=2$) For any $\vec L \in \mathbb R^2$ with $|\vec L| \geq n ^{1/2 + \varepsilon}$
$$\nu(x \in X: S_{(\check{\bm \varphi}, \tau)}(n,x) - n (0, \nu(\tau)) \in \vec L + [0,1]^2 ) \leq 
C \frac{n^{\beta - 2+ \varepsilon'}}{(\min \{ |\vec L| ,n\})^{2(\beta -1)}}.
$$
\end{lemma}

\begin{proof}
We prove the case $d=1$ and omit the similar proof for case $d=2$.
Recall that a cylinder of length $k$ a set of the type
$$\Delta^0_{i_1,...,i_k} = \bigcap_{j=1}^k T^{-j +1} \Delta_{0, i_j}.
$$
By bounded distortion, there exists a constant $C$ so that for any $k,l$, for any length $k$ cylinder
$\mathcal C_1$ and length $l$ cylinder $\mathcal C_2,$ and any $m \geq k$, 
\begin{equation}
\label{quasiind}
\nu (\mathcal C_1 \cap T^{-m} \mathcal C_2) \leq C \nu  (\mathcal C_1)
\nu  (\mathcal C_2)
\end{equation}
Given cylinders $\mathcal C_1$ of length $\lfloor n/2 \rfloor$ and $\mathcal C_2$ of length 
$\lceil n/2 \rceil$ (for ease of notation we drop the integer parts)
we say that $\mathcal C_1$ and $\mathcal C_2$ are compatible if there is some 
$x \in \mathcal C_1 \cap T^{ n/2 } \mathcal C_2$ so that 
$S_{\tau}(n,x) - n \nu(\tau) \in [L, L+1]$.
By \eqref{quasiind}, 
$$\nu(x \in X: S_{\tau}(n,x) - n \nu(\tau) \in [L, L+1] ) \leq 
\sum_{\mathcal C_1, \mathcal C_2} C \nu(\mathcal C_1) \nu(\mathcal C_2)
$$
where the sum
is taken over pairs of compatible cylinders 
$(\mathcal C_1, \mathcal C_2).$ 

We claim that if $\cC_1$ and $\cC_2$ are compatible then
for some $i = 1,2$, 
\begin{equation}
\label{highcyl}
|S_{\tau} (n/2, x)  - n \nu(\tau) / 2| \geq L/2 - \| \tilde \tau\|_{\kappa}/(1-\kappa) \text{ for all } x \in \mathcal C_i.
\end{equation}
Indeed, \eqref{highcyl} holds if there is some
$x \in \mathcal C_i$ with $|S_{\tau} (n/2, x)  - n \nu(\tau) / 2| \geq L/2$. Let us assume
$i=1$ (the case of $i=2$ is similar). 
By Lemma \ref{lemmaMel},
$$
\sum_{\mathcal C_1:\; \mathcal C_1 \text{ satisfies \eqref{highcyl}}} \nu (\mathcal C_1) \leq C
\frac{n^{\beta- 1 + \varepsilon'}}{(\min \{ L,n\})^{2 (\beta -1)}}.
$$
By Lemma \ref{lemtowerLLT}(B), for any $\mathcal C_1$ fixed,
$$
\sum_{\mathcal C_2 \text{ compatible with } \mathcal C_1} \nu(\mathcal C_2) \leq C n^{-1/2}.
$$
The lemma follows.
\end{proof}

\begin{lemma}
\label{lemmaLSLD}
 ({\em Local superlarge deviations}) 
 There is a constant $C$ so that 
 
 \noindent
($d=1$) For any $L$ with $|L| > 2 \nu(\tau) n $,
$$\nu(x \in X: S_{\tau}(n,x)  \in [L, L+1] ) \leq 
C \frac{n^{1/2}}{L^2}.
$$
($d=2$) For any $\vec L \in \mathbb R^2$ with $|\vec L| > 2 \nu(\tau) n$
$$\nu(x \in X: S_{(\check{\bm \varphi}, \tau)}(n,x) \in \vec L + [0,1]^2 ) \leq 
 \frac{C}{L^2}.
$$
\end{lemma}

\begin{proof}
We have by Chebyshev inequality that 
$$\nu(x \in X: |S_{\tau}(n,x)| \geq L ) \leq C \frac{n}{L^2} \text{ and }
\nu(x \in X: |S_{(\check{\bm \varphi}, \tau)}(n,x) | \geq |\vec L|  ) \leq 
 C \frac{n}{L^2}.
 $$
 The derivation of the local bound from the above global bound is the same as in the proof
 of Lemma \ref{lemmaLLD}.
\end{proof}

\begin{proof}[Proof of Theorem \ref{propYoungpoly}]
{\bf (H1)} and {\bf (H6)} are proved by Lemma \ref{lemtowerLLT} (A). 
In order to prove {\bf (H2)} and {\bf (H7)}, we decompose the
sum as
\begin{align*}
&S_1 = \sum_{n=1}^{w/3}, \quad
S_2 = \sum_{n = w/3}^{ w - w^{\gamma}},   \quad
&S_3 = \sum_{n: |n-w| \in \left[ K \sqrt w, w^{\gamma} \right]}, \quad
S_4 =  \sum_{ n = w + w^{\gamma}}^{5w/3},\quad
S_5 =  \sum_{n = 5w/3}^{\infty}
\end{align*}
where $\gamma = \min \{ \beta /4, 3/4\}$.
We need to show that for $d=1$, $f_X = \tau$ and for
$d=2$, $f_X = (\check{\bm \varphi}, \tau) $, that
\begin{equation}
\label{sums6}
\limsup_{w \rightarrow \infty} w^{\frac{d-1}{2}}\sum_{i=1}^6 S_i = o_{K\to\infty}(1).
\end{equation}

By Lemma \ref{lemmaLSLD}, we have 
$$
S_1 \leq C w^{-2} \sum_{n=1}^{w/3} n^{1-\frac{d}{2}} \leq C w^{-\frac{d}{2}} \ll w^{\frac{1-d}{2}}.
$$

We use Lemma \ref{lemmaLLD} with $L \lesssim n$ to conclude
$$ w^{\frac{d-1}{2}}(S_2+S_4)\leq  C w^{\frac{d-1}{2}} \sum_{m=w^{\gamma}}^{2w/3} 
C w^{\beta-1-\frac{d}{2}+\eps'} m^{-2\beta+ 2} \leq
C w^{\beta-\frac{3}{2} + \varepsilon' + \gamma (-2 \beta +3)}$$
which is $o(1)$ for sufficiently small $\varepsilon' = \varepsilon' (\beta)$ since $\gamma > 1/2$.

Next, we use Lemma \ref{lemtowerLLT} (B) to estimate $S_3$. Namely
 \begin{align*}
w^{\frac{d-1}{2}}S_3 &\leq C w^{\frac{d-1}{2}} \sum_{m = K \sqrt w}^{ w^{\gamma}}
\left( w^{-\frac{d}{2}} e^{-c m^2 /w}
+ w^{-\frac{d+\beta}{2} + 1} + w^{-\frac{d+1}{2}} \right) \\
& \leq C' e^{-c K} + C'w^{\gamma - \beta / 2 + 1/2} + C'w^{\gamma -1}.
\end{align*}
The above expression is $o_{K\to\infty, w\to\infty} (1)$ by the choice of $\gamma$.
Finally, we use Lemma \ref{lemmaLLD} with $L \gtrsim n$ to estimate $S_5$:
$$
S_5 \leq \sum_{n=5w/3}^{\infty} C n^{- \beta+ 1 - d/2 +\eps}
\leq C w^{- \beta+ 2 - d/2 +\eps} \ll w^{\frac{1-d}{2}}.
$$
We have verified \eqref{sums6}
and thus finished the proof of Theorem \ref{propYoungpoly}.
\end{proof}

\subsection{Proof of Proposition \ref{propunbdH}}
As mentioned earlier, under the conditions of Proposition \ref{propunbdH}, Theorem \ref{propYoungpoly}
implies 
{\bf (H1)}--{\bf (H7)} for $(\Phi, \bm{\varphi})$. Thus, by Theorem \ref{thm1}
and by \eqref{eq:factor2}, a {\it weaker} version of the MLCLT follows for $(\Psi, \bm{\chi})$: namely, when 
only those bounded and continuous test functions $\mathfrak X,\mathfrak Y: 
\aleph \rightarrow \mathbb R$ are allowed for which the corresponding
lift-ups $\mathfrak V = \mathfrak X \circ \iota, \mathfrak W= \mathfrak Y \circ \iota$
are supported on $\{ (x,s) \in \Omega: s \leq M\}$ for some finite $M$.

In order to complete the proof of Proposition \ref{propunbdH}, we need a
{\it stronger} version of the MLCLT for $(\Phi, \bm{\varphi})$: namely, all bounded and continuous
test functions $\mathfrak V, \mathfrak W$ on $\Omega$ are allowed. 
Recall that
$N_{t+s} = \max \{ n: S_{\tau}(n,x) < t+s\}.$
Approximating $\mathfrak V$ and $\mathfrak W$ by $\mathfrak V (x,s) 1_{s \leq M }$ and $\mathfrak W (x,s) 1_{s \leq M },$ 
we see that it suffices to show that 

\begin{align}
\lim_{M \rightarrow \infty} \lim_{t \rightarrow \infty} t^{1/2}
\mu \Bigg\{ (x,s): &(r(x) > M \text{ or } r(T^{N_{t+s} (x)}(x)) >M ) \nonumber \\
&\text{ and }
 \int_0^t \bm \varphi(\Phi^{s'}(x,s))ds'   - W(t) \in [0,1]
\Bigg\} =0, \label{eq:unbdH1}
\end{align}
uniformly for $W(t)$ satisfying \eqref{condMLCLT}. 
Recall that
$r_k=r(\Delta_{0,k})$ is the height of the tower above $\Delta_{0,k}$.
Consider the partition $\{ \Delta_{0,k}\}_{k \geq 1}$of $\Delta_0.$ 
Note that 
\begin{align*}
\mu((x,s): r(x) > t^{1/2 - \varepsilon}) &\leq \|\tilde{\tau}\|_{\infty}
\sum_{k: r_k \geq t^{1/2 - \varepsilon}} r_k \nu(\Delta_{0,k}) \\
& \leq C  \left[ t^{1/2 - \varepsilon}\nu (r \geq  t^{1/2 - \varepsilon}) + \sum_{m \geq  t^{1/2 - \varepsilon}} \nu(r \geq m) \right] \\
&  \leq C t^{(1/2 - \varepsilon )(1- \beta)} \ll t^{-1/2}
\end{align*}
assuming that $\epsilon< 1/2 - 1/(2 \beta -2)$.
Thus with the notations 
$$\mathcal I = \left\{  \int_0^t \bm \varphi(\Phi^{s'}(x,s))ds'   - W(t) \in [0,1]\right\}
\text{ and }\Delta_{0, \leq M} = \{ x \in \Delta_0: r(x) \leq M\},$$
it suffices
to verify that 
$$
\lim_{M \rightarrow \infty} \lim_{t \rightarrow \infty} t^{1/2}
(S_1 + S_2 + S_3) = 0,
$$
where
\begin{align*}
S_1 &= \sum_{k: r_k \in [M, t^{1/2 -\varepsilon}]}  
\mu((x,s)\in \mathcal I: x \in \Delta_{0,k}, T^{N_{t+s}(x)}(x) \in \Delta_{0,\leq M})\\
S_2 &= \sum_{l: r_l \in [M, t^{1/2 -\varepsilon}]}
\mu((x,s) \in \mathcal I: x \in \Delta_{0,\leq M}, T^{N_{t+s}(x)}(x) \in \Delta_{0,l})\\
S_3 &= \sum_{k: r_k \in [M, t^{1/2 -\varepsilon}]}  \sum_{l: r_l \in [M, t^{1/2 -\varepsilon}]}
\mu((x,s)\in \mathcal I: x \in \Delta_{0,k}, T^{N_{t+s}(x)}(x) \in \Delta_{0,l}).
\end{align*}

Using Lemma \ref{lemmaMel}(C) and proceeding the same way as in the proof of 
Theorem \ref{thm1}
there exits $C>0$ such that for any $t >0$, any $l$
with $r_l \leq t^{1/2 - \varepsilon}$, any $v \in \mathbb R$ and any $s \in [0, t^{1/2}]$
\begin{equation}
\label{eq:upperbdcyl}
\nu \left( x: \int_0^{N_{t+s}(x)} \bm \varphi (\Phi^{s'} (x,0)) ds' \in [v, v+1], 
T^{N_{t+s}(x)}(x) \in \Delta_{0,l} \right) 
\end{equation}
$$ \leq C r_l \nu(\Delta_{0,l}) t^{-1/2} .$$
We prove that
\begin{equation}
\label{eq:S3}
\lim_{M \rightarrow \infty} \lim{\sup}_{t \rightarrow \infty} t^{1/2} S_3 = 0,
\end{equation}
the proofs for $S_1$ and $S_2$ are similar and shorter. 
Pick some $k$ with $r_k \in [M, t^{1/2 - \varepsilon}]$. 
In particular, $|\int_0^{\tau(x)} \bm \varphi (\Phi^{s'} (x,0)) ds'| \leq C t^{1/2 -\varepsilon}$ for $x \in \Delta_{0,k}$.
Since the bijection 
$T^{r_k}: \Delta_{0,k} \rightarrow \Delta_0$ has bounded distortion,
\eqref{eq:upperbdcyl} implies
\begin{eqnarray*}
\nu \left( x \in \Delta_{0, k}: \int_{\tau (x)}^{N_{t+s}(x)} \bm \varphi (\Phi^{s'} (x,0)) ds' \in [v, v+1], 
T^{N_{t+s}(x)}(x) \in \Delta_{0,l} \right) \\
 \leq C r_l \nu(\Delta_{0,k}) \nu(\Delta_{0,l}) t^{-1/2}.
\end{eqnarray*}
Next, observe that 
$$
x \in \Delta_{0, k},\quad \int_{0}^{t} \bm \varphi (\Phi^{s'} (x,0)) ds' - W(t) \in [0,1], \quad
T^{N_{t+s}(x)}(x) \in \Delta_{0,l}
$$
imply 
$$
\left| \int_{\tau (x)}^{N_{t+s}(x)} \bm \varphi (\Phi^{s'} (x,0)) ds' - W(t) \right| \leq C (r_k + r_l).
$$
Thus $S_3$ can be bounded from above by
\begin{eqnarray*}
&&\sum_k \sum_l 
\mu \left( (x,s):
\begin{cases}
x \in \Delta_{0, k},\\
\left| \int_{\tau (x)}^{N_{t+s}(x)} \bm \varphi (\Phi^{s'} (x,0)) ds' - W(t) \right| \leq C (r_k + r_l),\\
T^{N_{t+s}(x)}(x) \in \Delta_{0,l}
\end{cases}
\right) \\
&&  \leq C\sum_{k: r_k \in [M, t^{1/2 -\varepsilon}]}  \sum_{l: r_l \in [M, t^{1/2 -\varepsilon}]} r_k r_l (r_k + r_l)  
\nu(\Delta_{0,k}) \nu(\Delta_{0,l}) t^{-1/2}\\
&&\leq C M^{3 - 2 \beta} t^{- 1/2}.
\end{eqnarray*}
We have verified \eqref{eq:S3} and finished the proof of Proposition \ref{propunbdH}. \qed

\section{Hyperbolic Young towers}
\label{ScHYT}

Let $\bm F: \bm M \rightarrow \bm M$ be a $\mathcal C^{1 + \varepsilon}$ diffeomorphism
of the Riemannian manifold $\bm M$. Assume that $\bm F$ satisfies assumptions (Y1)--(Y5) 
in Section 4.1 of \cite{PSZ17}.
These imply that there exists a "hyperbolic Young tower", namely a dynamical system $(\hat \Delta, \hat \nu, \hat F)$ which satisfies the following.
\begin{itemize}
\item
The base of the tower is the set
$\hat \Delta_0 = \hat \Delta_0^u \times \hat \Delta_0^s.$ The sets of the form
$A \times \hat \Delta_0^s$, $A \subset \hat \Delta_0^u$ be called
u-sets (similarly, sets of the form $\hat \Delta_0^u \times B$, $B \subset \hat \Delta_0^s$ are called s-sets).
Also, sets 
of the form $\Delta_0^u \times \{ x^s\}$ are called unstable manifolds and sets of the form $\{ x^u\} \times  \Delta_0^s$ are stable manifolds.
\item 
There is a partition of $\hat \Delta_0$ into s-sets $\hat \Delta_{0,k}=\hat \Delta^u_{0,k} \times \hat \Delta^s_{0}$ and positive integers $r_k$ so that
$\hat \Delta = \cup_{k\in \mathbb Z_+} \cup_{l=0}^{r_k -1} \hat \Delta_{l,k}$,
where $\hat \Delta_{l,k} = \{ (x,l): x \in \hat \Delta_{0,k}\}$.
\item For all $k$ and all $l =0,...,r_k-2$, $\hat F$ is an isomorphism between $\hat \Delta_{l,k}$ and $ \hat \Delta_{l+1,k}$ 
and $\hat F(x,l) = (x, l+1)$. Also
$\hat F$ is an isomorphism between $\hat \Delta_{r_k-1,k}$ and $\hat F(\hat \Delta_{r_k-1,k})$, the latter being a u-set of $\hat \Delta_0$.
Furthermore, if $x_1$ and $x_2$ belong to the same (un)stable manifold, so do $\hat F^{r_k}(x_1,0)$
and $\hat F^{r_k}(x_2,0)$. We write $\hat T = \hat F^{r_k - l}$ on $\hat \Delta_{l,k}$ and
$r(x^u, x^s) = r(x^u) = r_k$ for $(x^u, x^s) \in \hat \Delta_{0,k}$.
\item There is a mapping $\pi: \hat \Delta \rightarrow \bm M$ with $\pi|_{\hat \Delta_0}: \hat \Delta_0 \rightarrow \Lambda$ being a bijection,
where $\Lambda$ is a set with hyperbolic product structure and $\pi \circ \hat F = \bm F \circ \pi$.
\item Let $\Xi$ be the function on $\hat \Delta$ defined by $\Xi((x^s,x^u),l) = ((x^u, \bm x^s),l)$
with a fixed $\bm x^s$. Let $\tilde \Delta = \Xi (\hat \Delta)$ and 
$\tilde \nu = \Xi_* \hat \nu$. By the previously listed properties of $\hat F$, there is a well defined $\tilde F: \tilde \Delta \rightarrow \tilde \Delta$
such that $\Xi \circ \hat F = \tilde F \circ \Xi$. The dynamical system $(\tilde \Delta, \tilde \nu, \tilde F)$,
is an expanding
Young tower,  in the sense that it satisfies assumptions 
(A.1)--(A.4) of Section \ref{sec:towerpolydef}.
\item There exists some $a \in (0,1)$ so that for any $x,y$ on the same stable manifold in $\hat \Delta_0$, 
$d(\pi(\hat T(x)), \pi(\hat T(y))) < a d(\pi(x), \pi(y))$. Furthermore, for any $k$ and any $x,y$ on the same unstable manifold in $\hat \Delta_{0,k}$,
$d(\pi(x), \pi(y)) < a d(\pi(\hat T(x)), \pi(\hat T(y))) $
\end{itemize}

We also require that 
\begin{enumerate}
\item[(B1)] the expanding tower $(\tilde \Delta, \tilde \nu, \tilde F)$ satisfies \eqref{TailPoly} (this is slightly stronger requirement
than the ones in Section 4.1 of \cite{PSZ17}).
\item[(B2)] there are $K<\infty$ and $\hat\theta<1$ such that 
for every $k$, every $x,y \in \hat \Delta_{0,k}$ 
on the same stable manifold
  and every $0 \leq j \leq r_k -1$,
$$d(\pi(\hat F^j(x)), \pi(\hat F^j(y))) < K d(\pi(x), \pi(y)) \hat\theta^j . $$
\end{enumerate}

Consider the dynamical system $(\bm M, \lambda, \bm F)$ where $\lambda := \pi_* \hat \nu.$
Let $\upsilon$ be a positive H\"older roof function,
$\Psi^t$ is the corresponding suspension flow on the phase space $\aleph$. Let $\bm \chi$ be a zero mean continuous
observable so that
$\check{\bm \chi}(x)=\int_0^{\upsilon(x)} \bm\chi(\Psi^s x) ds$  
is H\"older.

Define $\hat \zeta: \hat \Delta \rightarrow \mathbb R_+$ by $\hat \zeta(x) = \upsilon (\pi (x))$. 
Let $\mathcal T$ be the phase space of
the suspension flow over $(\hat \Delta, \hat \nu, \hat F)$ with roof function $\hat \zeta$ 
and let $\hat{\bm \eta}: \mathcal T \rightarrow \mathbb R$
be defined by $\hat{\bm \eta} (x,s) = \bm \chi (\pi(x), s)$. Now we regard this flow as a suspension over the first return to the
base of the hyperbolic tower: let $\hat \tau (x) = \sum_{j=0}^{r(x) -1} \hat \zeta (\hat F^j (x))$, $\hat \Phi^t$ is the suspension
over $(\hat \Delta_0, \hat \nu|_{\hat \Delta_0}, \hat T)$ with roof function $\hat \tau$. Let the phase space of $\hat \Phi^t$ be denoted
by $\hat \Omega$ and its invariant measure be $\hat \mu = \hat \nu|_{\hat \Delta_0} \otimes Leb$.
We consider the observable $\hat {\bm \varphi}: \hat \Omega \rightarrow \mathbb R_+$ defined by 
$\hat {\bm \varphi}(x,s) = \hat \eta (x,s)$ (mind the identification $(x, \hat \zeta(x)) = (\hat F(x), 0) \in \mathcal T$).

We also introduce the suspension over the first return to the base
of the expanding tower $(\tilde \Delta, \tilde \nu, \tilde F)$. A fixed unstable manifold $\gamma^u$ is identified with the base
of the expanding tower, i.e. $\gamma^u$ is fully crossing $\Lambda$ in $\bm M$
and $\gamma^u \cap \Lambda = \pi \tilde \Delta_0$. 
Accordingly, any point $(x^u,l) \in \tilde \Delta$ is identified with $\bm F^l (\pi (x_u))$ and with $((x^u,\bm x^s),l) \in \hat \Delta$,
where $\{\bm x^u \} \times \hat \Delta_0^u$ is the lift-up of the unstable manifold $\gamma^u$ from $\Lambda$ to $\hat \Delta_0$. 
Let $T$ be the first return to the base, i.e. $T(x^u) = \sum_{j=0}^{r(x^u)-1} \tilde F(x^u)$.
The first return dynamics is $(X , \nu, T)$ where
$X = \tilde \Delta_0$ and $ \nu = \tilde \nu|_X.$
We consider the suspension flow over $(X, \nu, T)$
with roof function $\tau (x^u) = \sum_{j=0}^{r(x^u) -1} \upsilon (\bm F^j (\pi(x^u)))$. Denote this flow by $\Phi$, 
its phase space by $\Omega$ and consider 
the observable $\bm \varphi(x^u,s) =  \bm{\chi} (\pi(x^u),s )$ on $\Omega$ (with the identification $(x, \upsilon(x)) = (\bm F(x), 0) \in \aleph$).

By the above constructions and by (B1), we can apply Theorem \ref{propYoungpoly} and Propositions \ref{corpoly1}
to $(\Phi, \bm \varphi)$. 
However, if we want to conclude some result corresponding to Proposition \ref{propunbdH}, we need to
extend the MLCLT and the moderate deviation estimates from $(X, \nu, T)$ to $(\hat X = \hat \Delta_0, \hat \nu|_{\hat \Delta_0}, \hat T)$.
In order to do so, we introduce the metric $d_{\hat X}$ on $\hat X$. Let $d_{\hat X}(x,y) = 1$ if $x$ and $y$ belong to different
partition elements $\hat \Delta_{0,k}$. If $x = (x^u,x^s), y= (y^u,y^s) \in \hat \Delta_{0,k}$, then 
$d_{\hat X}(x,y) = d(\pi(\bm x^u, x^s), \pi(\bm x^u, y^s)) + \beta^{s(x^u, y^u)}$ with a fixed $\bm x^u$, where $s$ is the separation time
as in (A3) and $\beta <1$.
Let $\phi: \bm F \rightarrow \mathbb R^d$ be a piecewise H\"older function, $\hat \phi: \hat \Delta \rightarrow \mathbb R^d$
is its lift-up (i.e. $\hat \phi(x) = \phi (\pi (x))$). We also define $\tilde \phi: \tilde \Delta \rightarrow \mathbb R^d$
by $\tilde \phi (x^u, l) = \hat \phi((x^u, \bm x^s),l)$, $\hat \psi: \hat X \rightarrow \mathbb R^d $
by $\hat \psi (x) = \sum_{j=0}^{r(x)-1} \hat \phi(x,j)$ and 
$\tilde \psi: X \rightarrow \mathbb R^d $, $\tilde \psi (x^u) = \sum_{j=0}^{r(x^u)-1} \tilde \phi(x^u,j)$.
We shall use the following standard fact.

\begin{lemma}
\label{lemhyp}
There is a H\"older function $h: \hat X \rightarrow \mathbb R^d$ such that 
\begin{equation}
\label{hyph}
\hat \psi(x^u, x^s) = \tilde \psi (x^u) + h( x^u, x^s) - h (\hat T(x^u, x^s))
\end{equation}
\end{lemma}

\begin{proof}
Let
$
\displaystyle h( x^u, x^s) = \sum_{m=0}^{\infty} \hat \psi(\hat T^m (x^u, x^s)) -  \hat \psi(\hat T^m (x^u, \bm x^s)).
$
It is straightforward to verify that $h$ is H\"older using (B2). Equation \eqref{hyph} also
follows by a direct computation. 
\end{proof}

By Theorem \ref{propYoungpoly} implies that $(X, \nu, T)$ satisfies MLCLT.
Lemma \ref{lemhyp} and continuous mapping theorem allow to lift the MLCLT to
$(\hat \Delta_0, \hat \nu|_{\hat \Delta_0}, \hat T).$
Also since the ergodic sums of $\hat\psi$ and $\tilde\psi$ differ by $O(1)$ 
the moderate and large deviation estimates for $\hat\psi$ follow from the corresponding estimates
for $\tilde\psi.$ Proceeding as in Section \ref{secpoly} we obtain:

\begin{proposition}
\label{propunbdHhyp}
In the above setup let us also assume that 
$(\check{\bm \chi}, \upsilon)$ is non-arithmetic. Then for any continuous $\mathfrak X, \mathfrak Y: \aleph \rightarrow \mathbb R$,
any continuous and compactly supported $\mathfrak Z: \mathbb R \rightarrow \mathbb R$,
and any $W(t)$ with $W(t) / \sqrt t \rightarrow W$,
\begin{align}
&\lim_{t \rightarrow \infty} t^{1/2} \int_{\aleph} \mathfrak X (x,s) \mathfrak Y( \Psi^t (x,s))
\mathfrak Z \left( \int_0^t \bm \chi(\Psi^{s'}(x,s))ds'   - W(t) \right) d\kappa(x,s) \nonumber \\
&= \mathfrak g_{\Sigma} (W) 
\int \mathfrak X d \kappa \int \mathfrak Y d \kappa \int \mathfrak Z d Leb .
\nonumber
\end{align}
\end{proposition}

\section{Examples}
\label{secex}

\subsection{iid random variables (reward renewal processes)}
\label{SS:iid}

Let $\mathbb P$ be a probability measure supported on a compact subset of $\mathbb R \times \mathbb R_+$.
Assume that $\int x_1 d \mathbb P(x_1,y_1) =0$ and that the minimal translated groups supporting the measure 
$\mathbb P_2$ defined by
$\mathbb P_2 (A) := \mathbb P (\mathbb R \times A)$ is either $\mathbb R$ or $r + \alpha \mathbb Z$ with $r / \alpha \notin \mathbb Q$.
Let $X= (\mathbb R \times \mathbb R_+)^{\mathbb Z_+}$ equipped with the product topology, $T: X \rightarrow X$
is the left shift and
$\nu = \mathbb P ^{{\otimes}\mathbb Z_+}$.
For any $\underline x = ((x_1,y_1),(x_2,y_2),...) \in X$, let 
$\tau(\underline x)= y_1$ and $\bm \varphi (\underline x, t) = \varphi ((x_1,y_1),t)$ be a continuous function (where
$t \in [0, y_1]$), which also satisfies $\check{\bm \varphi} (\underline x) = x_1$.
Then {\bf (H1)} - {\bf (H7)}
can be proved by a much simplified version of our proof of Theorem \ref{propYoungpoly} based on 
classical results in probability theory. Namely, {\bf (H1)} and {\bf (H6)} hold by
\cite{R62} and \cite{S65}. The proof of Lemma~\ref{lemtowerLLT} also extends since
\cite{R62} and \cite{S65} use the Fourier method. Finally, the moderate deviation estimates follow from 
e.g. Theorem 5.23 in \cite{P75}, Chapter~5. Consequently the results of
Section \ref{sec:mainres} apply to iid random variables.

This example could be used to illustrate that the MLCLT does not always hold for suspension flows.

Let $(X_1,Y_1)$ be a random vector that can take the following three values, 
all with
probability $1/3$: $(-1,2-\sqrt 2)$, $(0,1)$, $(1, \sqrt 2 -1)$. Let $(X_i, Y_i)$ be iid. 
Let $t_n = \sum_{i=1}^{n} Y_i$,
$N_{t} = \max \{ n: \sum_{i=1}^{n} Y_i \leq t\}$ and $S_n = \sum_{i=1}^n X_i$. Then it is well known that
$S_{N_{t}}$ satisfies the CLT as well as many other limit theorems (see e.g.  \cite{GW93}) but we claim that the 
it does not satisfy the LCLT. 
Indeed, it is easy to check that 
$S_{N_{t}} = 0 $ implies that the last renewal time before $t$ is $\lfloor t \rfloor$ and $S_{N_{\lfloor t \rfloor}} =0$.
Thus for $K \gg 1$ positive integers,
$\mathbb P (S_{N_{K + a}} = 0) \sim \mathbb P (S_{N_{K+b}} = 0)$
if and only if $\lfloor a \rfloor$ and $\lfloor b \rfloor$ fall into the same partition element of 
$$\{ [0, \sqrt 2 -1), [\sqrt 2 -1, 2 - \sqrt 2), [2 - \sqrt 2, 1)\}.$$
This shows that $\displaystyle \lim_{t \rightarrow \infty} t^{1/2} \mathbb P (S_{N_t} = 0)$ only exists along subsequences.

The above example fits into the abstract framework at the beginning of Example
\ref{SS:iid}. Namely,
$\mathbb P$ is the uniform measure supported on 
the three points $(-1,2-\sqrt 2)$, $(0,1)$, $(1, \sqrt 2 -1)$. Then the minimal group supporting $\mathbb P$,
i.e supporting the values of $(\check{\bm \varphi}, \tau)$ is the subgroup generated by $(0,1)$ and $(1, \sqrt 2)$;
and the translation is zero. Consequently, the linearized group is the same as the minimal group and we are in
Case {\bf (D)}. Note that this is not a generic example among probability measures on three atoms. Indeed, in the
generic case, the translation of the minimal lattice would not be rationally related to the lattice and hence the linearization
would give $\mathbb R^2$, i.e. Case {\bf (A)}. Thus Case {\bf (A)} is generic even among discrete distributions.

\subsection{Axiom A flows}
\label{SSAxiomA}

Let $\Psi^t$ be a $\mathcal C^2$ Axiom A flow,
which is topologically transitive on a locally
maximal hyperbolic set $\Lambda$.
Then Bowen \cite{B73} and Bowen and Ruelle \cite{BR75} proved that
there exists
a topologically mixing subshift of finite type
$(\Sigma_A, \sigma)$ and a positive H\"older
roof function $\tau: \Sigma_A \rightarrow \mathbb R_+$
such that for the corresponding
suspension flow $\Phi^t: \Omega \rightarrow \Omega$
and for a suitable Lipschitz continuous surjection
$\rho: \Sigma_A \rightarrow \Lambda$, the following diagram commutes:
\[ \begin{tikzcd}
\Omega \arrow{r}{\Phi^t} \arrow[swap]{d}{\rho} & \Omega \arrow{d}{\rho} \\%
\Lambda \arrow{r}{\Psi^t}& \Lambda
\end{tikzcd}
\]
Let $\lambda$ be an equilibrium measure 
on $\Lambda$ with H\"older potential $G.$ Then $\rho$ is a measure theoretic
isomorphism
between $(\Omega, \mu, \Phi^t)$ 
and $(\Lambda, \lambda, \Psi^t),$
where $\mu$ is a Gibbs measure with potential $\bar{G}=G\circ \rho.$
 By general theory, $\mu = \nu \otimes Leb /\nu(\tau)$,
where $\nu$ is a measure on $\Sigma_A$, invariant
under $\sigma$ (the equilibrium state of $\check{\bar G} - P(G) \tau$,
where $P$ is the pressure).
Thus the MLCLT for $(\Psi^t, \lambda)$ is implied by the MLCLT for $(\Phi^t, \mu)$.
For the proof of the latter one, a simplified version of
Section \ref{secpoly} applies (e.g. {\bf (H2)} and {\bf (H6)}
follow from \cite{GH88} and \cite{G89}).
Thus we have the following analogue of Corollary \ref{corpoly1}.

\begin{proposition}
\label{cor:min2}
Consider $\Psi^t:\Lambda\to\Lambda$ be Axiom A flow, $\lambda$ be a Gibbs measure and 
$\psi:\Lambda\to\mathbb{R}$ be a H\"older observable. Denote $\bm \varphi=\psi\circ \rho.$
Assume that $(\check{\bm \varphi}, \tau)$
is minimal and its linearized group falls into cases {\bf (A)}, {\bf (B)} or {\bf (C)}. 
Then 
the conclusion of Theorem \ref{thm1} holds for $(\Psi,  \psi)$
with $h(x) = 0$ and $h_{\tau}(x) = 0$.
\end{proposition}

We mention that \cite{W96} essentially proves Proposition \ref{cor:min2}, case {\bf (A)} (note that the
flow-independence condition of \cite{W96} Theorem 2 implies case {\bf (A)} by Proposition 3 of \cite{W96}).

\subsection{Suspensions over Pomeau-Manneville maps}

Consider next Pomeau-Manneville maps 
(a.k.a. Liverani-Saussol-Vaienti maps). 
Namely, let $\mathbb M = [0,1]$ and $\mathbb  F: \mathbb M \rightarrow 
\mathbb M$ be
defined by
$$
\mathbb  F(x) =
\begin{cases}
x(1 + 2^{\alpha} x^{\alpha}) & \text{if } 0 \leq x \leq 1/2\\
2x -1  & \text{if } 1/2 < x \leq 1.
\end{cases}
$$
Suppose that $\alpha<1.$
Then the first return map to $[\frac{1}{2}, 1]$ gives an expanding Young tower satisfying the assumptions
of Section \ref{secpoly} (see e.g. the discussion in Section 1.3 of \cite{G05}).
In particular,
$\mathbb F$ has a unique absolutely continuous
invariant probability measure $\lambda$ 
(with Lipschitz density on any compact subinterval of $(0,1]$, see \cite[Lemma 2.3]{LSV99}).
In addition \eqref{TailPoly} is satisfied with $\beta=1/\alpha.$ In particular, 
$\beta>2$ if $\alpha<\frac{1}{2}.$
Consider a H\"older roof function
$\upsilon$ on $\mathbb M$, and let $\Psi$ be the corresponding suspension semiflow on the phase space $\aleph$. 
Let $\bm \chi: \aleph \rightarrow \mathbb R$ be a zero mean continuous observable so that 
$\check{\bm \chi}$ is H\"older.  Applying Proposition \ref{propunbdH} we get

\begin{proposition}
MLCLT in valid for suspension semiflows over  Pomeau-Manneville maps with
H\"older roof functions provided that $\alpha<\frac{1}{2}$ and the pair $(\check{\bm \chi}, \upsilon)$ is non-arithmetic. 
\end{proposition}

We note that in case $\alpha >1$, the invariant measure $\lambda$ is infinite.
This case is discussed 
in \cite{DN17}. The approach of \cite{DN17} is somewhat similar to that of the present paper.

\subsection{Sinai billiard flows with finite horizon}
\label{SSSinai}

Let $\mathcal D = \mathbb T^2 \setminus \cup_{i=1}^I B_i$, where $\mathbb T^2$ is the $2$-torus
and $B_1, ..., B_I$ are disjoint strictly convex subsets of $\mathbb T^2$, whose boundaries
are $\mathcal C^3$ smooth with curvature bounded away from zero. The Sinai billiard flow 
$\Psi^t$ describes a point particle moving with unit speed  
in the interior of $\mathcal D$
and having specular reflection on $\partial \mathcal D$ (i.e. the angle of incidence equals the
angle of reflection). The phase space of $\Psi^t$ is thus 
$\aleph = \mathcal D \times \mathcal S^1$ (pre- and post-collisional points on $\partial \mathcal D \times \mathcal S^1$
are identified). Consider the measure $\kappa = c Leb_{\mathcal D} \otimes Leb_{\mathcal S^1}$, where
$c$ is a normalizing constant. We assume the {\it finite horizon} condition, i.e. that the 
time in between two collisions with $\partial \mathcal D$ is bounded.
(This assumption is natural since if the horizon is infinite, then the return time has infinite second
moment and a non-standard normalization is needed in the CLT (\cite{SzV07,CD09})).
We can regard $\Psi^t$ as a suspension flow over the {\it billiard ball map}: the Poincar\'e section on the
boundary of the scatterers. Namely, the billiard ball map is 
$\mathbb F : \mathbb M \rightarrow \mathbb M$, where 
$$\mathbb M = \{ x = (q, \bm v) \in \partial \mathcal D \times \mathcal S^1, \langle \bm v, n \rangle \geq 0 \},$$
where $n$ is the normal vector of $\partial \mathcal D$ at the point $q$ pointing inside $\mathcal D$
(post-collisional point) and $\mathbb F (x) = \Psi^{\upsilon(x)} (x)$, $\upsilon$ being the time needed until the next 
collision. Let the projection of $\kappa$ to $\mathbb M$ be denoted by $\lambda$ (then $\lambda$ is the SRB measure,
and it has density $c \cos (\bm v)$ with respect to the Lebesgue measure on $\mathbb M$).
Let $\bm \chi: \aleph \rightarrow \mathbb R$ be a piecewise H\"older observable and 
$\check{\bm \chi}: \mathbb M \rightarrow \mathbb R$ is defined by 
$\check{\bm \chi} (x) = \int_0^{\upsilon(x)} \bm \chi (x,s) ds$
as before.
For the dynamical system 
$(\mathbb M, \lambda, \mathbb F)$,
a tower with exponential tails was constructed in \cite{Y98}. 
Thus Propositions \ref{corpoly1} and \ref{propunbdHhyp} imply

\begin{proposition}
\label{propbilliard} 
Assume that $(\check{\bm \chi}, \upsilon)$ is minimal and its linearized
group falls into cases {\bf (A)}, {\bf (B)} or {\bf (C)}. 
Then 
the conclusion of Theorem \ref{thm1} holds for $(\Phi, \bm \varphi)$
with $h(x) = 0$, $h_{\tau}(x) = 0$.
Assume furthermore that
$(\check{\bm \chi}, \upsilon)$ is non-arithmetic. Then for any continuous $\mathfrak X, \mathfrak Y: \aleph \rightarrow \mathbb R$,
any continuous and compactly supported $\mathfrak Z: \mathbb R \rightarrow \mathbb R$,
and any $W(t)$ with $W(t) / \sqrt t \rightarrow W$,
\begin{align}
&\lim_{t \rightarrow \infty} t^{1/2} \int_{\aleph} \mathfrak X (x,s) \mathfrak Y( \Psi^t (x,s))
\mathfrak Z \left( \int_0^t \bm \chi(\Psi^{s'}(x,s))ds'   - W(t) \right) d\kappa(x,s) \nonumber \\
&= \mathfrak g_{\Sigma} (W) 
\int \mathfrak X d \kappa \int \mathfrak Y d \kappa \int \mathfrak Z d Leb.
\nonumber
\end{align}
\end{proposition}

One special case
of Proposition \ref{propbilliard} (Case {\bf (C)}) for the Sinai billiard flow
(namely, when $\bm \chi$ is the horizontal coordinate of the free flight function)
is analyzed in Section A.2 of \cite{DN16}.
We remark that although finding the minimal group of $(\check{\bm \varphi}, \tau)$ in general is not easy,
it is possible in some special cases such as the one studied in \cite{DN16}
(cf. \cite[Lemma A.3]{DN16}).

\subsection{Geometric Lorenz flows}
\label{SSLorenz}

Let us consider 
  a geometric Lorenz flow $T^t: \mathbb R^3 \rightarrow \mathbb R^3$
with the SRB measure $\hat \mu$ as defined in \cite{HM06} (in \cite{HM06} the SRB measure is denoted by $\mu$).
In \cite{HM06},
$T^t$ is represented as a suspension over a Poincar\'e map $P: X \rightarrow X$ with roof function $v$.
(This function is denoted by $h$ in \cite{HM06}).
Given H\"older observable $\bm \chi: \mathbb R^3 \rightarrow \mathbb R$,
let $\check{\bm \chi} = \int_0^{\upsilon (x)} \bm \chi (x,s) ds$.
Furthermore, in \cite{HM06},
a tower $(\Delta, F, \tilde \nu)$
satisfying assumptions (A1)--(A4). of Section \ref{sec:towerpolydef} is constructed.
(More precisely, as usual, the measure $\tilde \nu$ is not a priori given, but
it exists by  \cite{Y99} and then it is pulled back to $\mathbb R^3$.
  We note that slightly different earlier constructions
are also available in the literature. However we follow \cite{HM06} for simplicity,
so we do not review the earlier work here).
They also show that the roof function $f = \tilde \tau$ as well as any function  
$f= \tilde{\check{ \bm \chi}}$
corresponding to a given H\"older observable $\bm \chi: \mathbb R^3 \rightarrow \mathbb R$ 
satisfy
$$
|f(x,l)| \leq C r(x), \quad |f(x,l) - f(y,l)| \leq C r(x) \varkappa^{s(x,y)}
$$
(here, as usual, $(x,l) \in \Delta$). Although these functions are not in
$C_{\varkappa}(\Delta, \mathbb R)$, but we can easily "stretch" the tower by defining $(\Delta', F', \tilde \nu ')$
with $r(\Delta'_{0,k}) = r(\Delta_{0,k})^2$, $\tilde \nu ' ( \Delta'_{0,k}) = \tilde \nu( \Delta_{0,k})$.
The corresponding functions are $\tilde{\tau}' (x,l) := \tilde{\tau} (x,\lfloor l/r(x) \rfloor) / r(x)$,
$\tilde{\check{ \bm \varphi}}' (x,l) := \tilde{\check{ \bm \varphi}} (x,\lfloor l/r(x)\rfloor ) / r(x)$.
The first return maps to $\Delta_0$ and $\Delta'_0$ are clearly isomorphic just like the
suspension flows with base
$\Delta '$ and roof function $\tilde{\tau}'$ and with base
$\Delta $ and roof function $\tilde{\tau}$.
Furthermore, the functions $\tilde{\check{ \bm \varphi}}', \tilde{\tau}'$ are elements of
$C_{\varkappa}(\Delta', \mathbb R)$. Since $\tilde \nu (r> n)$ decays superpolynomially,
\eqref{TailPoly} holds with $\beta >2$ for $\Delta '$.
Thus Proposition 
\ref{propunbdHhyp} gives

\begin{proposition}
Let $T^t$ be geometric Lorenz flow. Assume that $(\check{\bm \chi},\upsilon)$
is non-arithmetic. 
Then for any continuous $\mathfrak X, \mathfrak Y: \mathbb R^3 \rightarrow \mathbb R$,
any continuous and compactly supported $\mathfrak Z: \mathbb R \rightarrow \mathbb R$,
and any $W(t)$ with $W(t) / \sqrt t \rightarrow W$,
\begin{align}
&\lim_{t \rightarrow \infty} t^{1/2} \int_{\mathbb R^d} \mathfrak X (x) \mathfrak Y( T^t (x))
\mathfrak Z \left( \int_0^t \bm \chi(T^{s}(x))ds   - W(t) \right) d\hat \mu(x) \nonumber \\
&= \mathfrak g_{\Sigma} (W) 
\int \mathfrak X d \hat \mu \int \mathfrak Y d \hat \mu \int \mathfrak Z d Leb.
\nonumber
\end{align}
\end{proposition}

Note that $\upsilon$
is non-arithmetic by \cite{LMP05}, but the non-arithmeticity of the pair $(\check{\bm \chi},\upsilon)$
is a non-trivial assumption.


\end{document}